\documentclass[reqno,a4paper,10pt]{amsart}
\usepackage{eucal,amsfonts,amssymb,amsmath,amsthm,epsfig,mathrsfs}
\usepackage{color}
\usepackage{amscd,amsxtra}
\usepackage{enumerate}
\usepackage{latexsym}

\newcommand{\opnorm}[1]{%
  \left\vert\mkern-1mu\left\vert\mkern-1mu\left\vert #1
    \right\vert\mkern-1mu\right\vert\mkern-1mu\right\vert}

\allowdisplaybreaks

\makeatletter \@addtoreset{equation}{section}

\makeatother \makeatletter

\newtheorem{thm}{Theorem}[section]
\newtheorem{hyp}[thm]{Hypotheses}{\rm}
\newtheorem{hypo}[thm]{Hypothesis}{\rm}
\newtheorem{lemm}[thm]{Lemma}
\newtheorem{coro}[thm]{Corollary}
\newtheorem{prop}[thm]{Proposition}

\newtheorem{rmk}[thm]{Remark}{\rm}

\newcommand{\R}{{\mathbb R}}
\newcommand{\N}{{\mathbb N}}
\newcommand{\Rd}{\mathbb R^d}
\newcommand{\A}{\mathcal{A}}
\newcommand{\one}{\mbox{$1\!\!\!\;\mathrm{l}$}}

\title[hypercontractivity, supercontractivity...]{Hypercontractivity, supercontractivity, ultraboundedness and stability in semilinear problems}
\thanks{The authors are members of GNAMPA of the italian Istituto Nazionale di Alta Matematica. This work has been supported by
the INdAM-GNAMPA Project 2016 ``Equazioni e sistemi di equazioni ellittiche e paraboliche associate ad operatori con coefficienti illimitati e discontinui''.}
\author{D. Addona}
\author{L. Angiuli}
\author{L. Lorenzi}
\address{D.A. \& L.L.: Dipartimento di Matematica, Universit\`a degli Studi di Parma, Parco Area delle Scienze 53/A, I-43124 Parma, Italy.}
\address{L.A.: Dipartimento di Matematica, Universit\`a del Salento, via per Arnesano, s.n., 73100 Lecce, Italy.}
\email{d.addona@campus.unimib.it}
\email{luciana.angiuli@unisalento.it}
\email{luca.lorenzi@unipr.it}
\keywords{nonautonomous second-order elliptic
operators, semilinear parabolic equations, unbounded coefficients, hypercontractivity, supercontractivity, ultraboundedness, stability}
\subjclass[2010]{35K58, 37L15}

\date{\today}

\begin{document}

\begin{abstract}
We study the Cauchy problem associated to a family of nonautonomous semilinear equations in the space of bounded and continuous functions over $\Rd$ and in $L^p$-spaces with respect to tight evolution systems of measures. Here, the linear part of the equation is a nonautonomous second-order elliptic operator with unbounded coefficients defined in $I\times \Rd$, ($I$ being a right-halfline).
To the above Cauchy problem we associate a nonlinear evolution operator, which we study in detail, proving some summability improving properties. We also study the stability
 of the null solution to the Cauchy problem.
\end{abstract}

\maketitle

\section{Introduction}

This paper is devoted to continue the analysis started in \cite{AL}. We consider
a family of linear second-order differential
operators $\mathcal{A}(t)$  acting on smooth function $\zeta$ as
\begin{align}\label{oper}
(\mathcal{A}(t)\zeta)(x)=\sum_{i,j=1}^d q_{ij}(t,x)D_{ij}\zeta(x)+
\sum_{i=1}^d b_i(t,x)D_i\zeta(x)
,\qquad\;\,t \in I,\,\, x\in \Rd,
\end{align}
where $I$ is either an open right halfline or the whole $\R$.
Then, given $T>s \in I$, we are interested in studying the nonlinear Cauchy problem
\begin{equation}
\left\{
\begin{array}{ll}
D_tu(t,x)=(\A(t)u)(t,x)+\psi_u(t,x), & (t,x)\in (s,T]\times\Rd,\\[1mm]
u(s,x)= f(x), &x \in \Rd,
\end{array}
\right.
\label{n-sm-pb}
\end{equation}
where $\psi_u(t,x)=\psi(t,x,u(t,x),\nabla_xu(t,x))$.
We assume that the coefficients $q_{ij}$ and $b_i$ ($i,j=1,\dots,d$), possibly unbounded, are smooth enough, the diffusion matrix $Q=[q_{ij}]_{i,j=1, \dots,d}$ is uniformly elliptic and there exists a
Lyapunov function $\varphi$ for $\mathcal{A}(t)$ (see Hypothesis \ref{base}(iii)).
These assumptions yield that the linear part $\mathcal{A}(t)$ generates a linear evolution operator $\{G(t,s):\, t\geq s\in I\}$ in $C_b(\Rd)$. More precisely, for every $f\in C_b(\Rd)$ and $s \in I$, the function $G(\cdot,s)f$ belongs to $C_b([s,+\infty)\times \Rd)\cap C^{1,2}((s,+\infty)\times\Rd)$,
it is the unique bounded classical solution of the Cauchy problem \eqref{n-sm-pb}, with $\psi\equiv 0$,
and satisfies the estimate
\begin{equation}\label{norm_g}
\|G(t,s)f\|_\infty \le \|f\|_\infty,\qquad\;\, t>s\in I,\;\, f \in C_b(\Rd).
\end{equation}

We refer the reader to \cite{KunLorLun09Non} for the construction of the evolution operator $G(t,s)$ and for further details.

Classical arguments can be adapted to our case to prove the existence of a unique local mild solution $u_f$ of problem \eqref{n-sm-pb} for
any $f\in C_b(\Rd)$, i.e., a function
$u:[s,\tau]\times \R^d\to \R$ (for some $\tau>s$) such that
\begin{align}\label{defi_mild}
u(t,x)= (G(t,s)f)(x)+\int_s^t(G(t,r)\psi_u(r,\cdot))(x)dr,\qquad\;\,t\in [s,\tau],\;\,x\in\Rd.
\end{align}

Under reasonable assumptions such a mild solution $u_f$ is classical, defined in the whole $[s,+\infty)$ and satisfies the condition
\begin{eqnarray*}
\|u_f\|_\infty+\sup_{t \in (s,T)}\sqrt{t-s}\|\nabla_x u_f(t, \cdot)\|_\infty<+\infty
\end{eqnarray*}
for any $T>s$.
Hence, setting ${\mathcal N}(t,s)f=u_f(t,\cdot)$
for any $t>s$ we deduce that ${\mathcal N}(t,s)$ maps $C_b(\Rd)$ into $C^1_b(\Rd)$ and, from the uniqueness of the solution to \eqref{n-sm-pb}, it follows that it satisfies the evolution law
${\mathcal N}(t,s)f={\mathcal N}(t,r){\mathcal N}(r,s)f$ for any $r\in (s,t)$ and $f\in C_b(\Rd)$.

As in the linear case we are also interested to set problem \eqref{n-sm-pb} in an $L^p$-context.
However, as it is already known from the linear case, the most natural $L^p$-setting where problems with unbounded coefficients can be studied is that related to the so-called {\it evolution systems of measures} (\cite{DaPRoc}),
that is one-parameter families of Borel probability measures $\{\mu_t: t \in I\}$ such that
\begin{equation}\label{inv_pro}
\int_{\Rd} G(t,s)fd\mu_t= \int_{\Rd} fd\mu_s,\qquad\;\,f\in C_b(\Rd),\;\,t>s \in I.
\end{equation}

When they exist, evolution families of measures are in general infinitely many, even the tight ones, where, roughly speaking, tight means that all the measures of the family are essentially concentrated on the same large ball
(see Section \ref{sect-2} for a rigorous definition of tightness).
Under additional assumptions on the coefficients of the operator $\mathcal{A}(t)$ (see Section \ref{sect-2}), there exists a unique tight evolution system of measures $\{\mu_t: t \in I\}$,
which has the peculiarity to be the unique system related to the asymptotic behavior of $G(t,s)$ as $t$ tends to $+\infty$.
We also mention that, typically, even if for $t\neq s$, the measures $\mu_t$ and $\mu_s$ are equivalent (being equivalent to the restriction of the Lebesgue measure to the Borel $\sigma$-algebra in $\Rd$), the corresponding $L^p$-spaces differ. In this paper we consider any
tight evolution system of measures.

Formula \eqref{inv_pro} and the density of $C_b(\Rd)$ in $L^p(\Rd, \mu_s)$ allow to extend $G(t,s)$ to a contraction from $L^p(\Rd, \mu_s)$ to $L^p(\Rd, \mu_t)$ for any $t>s$ and any $p \in [1, +\infty)$ and to prove very nice properties of $G(t,s)$ in these spaces.

In view of these facts, it is significant to extend $\mathcal{N}(t,s)$ to an operator from $L^p(\Rd,\mu_s)$ to $L^p(\Rd, \mu_t)$ for any $I\ni s<t$.
This can be done if $p\ge p_0$ (see Hypothesis \ref{base} (v)), $\psi(t,x, \cdot, \cdot)$ is Lipschitz continuous in $\R^{d+1}$ uniformly with respect to $(t,x)\in (s,T]\times\Rd$ and, in addition, $\sup_{t\in (s, T]}\sqrt{t-s}\|\psi(t,\cdot,0,0)\|_{L^p(\Rd, \mu_t)}<+\infty$.
In particular, each operator ${\mathcal N}(t,s)$ is continuous from $L^p(\Rd,\mu_s)$ to $W^{1,p}(\Rd,\mu_t)$.

We stress that the first condition on $\psi$ may seem too restrictive, but in fact it is not. Indeed, the Sobolev embedding theorems fail to hold, in
general, when the Lebesgue measure is replaced by any of the measures $\mu_t$. This can be easily seen in the particular case of the one-dimensional Ornstein-Uhlenbeck operator, where the evolution system of measures is replaced by a time-independent measure $\mu$ (the so-called invariant measure), which is the gaussian centered at zero with covariance $1/2$. For
any $\varepsilon>0$, the function $x\mapsto \exp(2(2p+\varepsilon)^{-1}|x|^2)$ belongs to $W^{k,p}(\R,\mu)$ for any $k\in\N$ but it does not belong to $L^{p+\varepsilon}(\Rd,\mu)$.

Under the previous assumptions, for any $f\in L^p(\Rd,\mu_s)$, ${\mathcal N}(\cdot,s)$ can be identified with the unique
mild solution to problem \eqref{n-sm-pb} which belongs to $L^p((s,T)\times \Rd,\mu)\cap W^{0,1}_p(J\times \Rd,\mu)$, for any $J\Subset (s,T]$, such that $u_f(t,\cdot)\in W^{1,p}(\Rd,\mu_t)$ for almost every $t\in (s,T]$.
Here, $\mu$ is the unique Borel measure on the $\sigma$-algebra of all the Borel subsets
of $I\times\R^d$ which extends the map defined on the product of a Borel set $A\subset I$ and a Borel set $B\subset\Rd$ by
\begin{equation*}
\mu(A \times B):=\int_A\mu_t(B)dt.
\end{equation*}

Since, as it has been stressed, in this context the Sobolev embedding theorems fail to hold in general, the summability improving properties of the nonlinear evolution operator $\mathcal{N}(t,s)$ are not immediate and true in all the cases.
For this reason in Section \ref{4} we investigate properties such as hypercontractivity, supercontractivity, ultraboundedness of the evolution operator $\mathcal{N}(t,s)$ and its spatial gradient.
Differently from \cite{AL}, where $\psi=\psi(t,u)$ and the hypercontractivity of $\mathcal{N}(t,s)$ is proved assuming $\psi(t,0)=0$ for any $t>s$, here we consider a more general case. More precisely we assume that there exist $\xi_0\ge 0$ and $\xi_1,\xi_2 \in \R$ such that
$u\psi(t,x,u,v)\le \xi_0|u|+\xi_1u^2+\xi_2|u||v|$
 for any $t\ge s$, $x,v\in\Rd$ $u\in\R$.
Under some other technical assumptions on the growth of the coefficients $q_{ij}$ and $b_i$ $(i,j=1,\ldots,d)$ as $|x|\to +\infty$, we show that
as in the linear case, (see \cite{AngLorOnI, AngLorLun}), the hypercontractivity and the supercontractivity of $\mathcal{N}(t,s)$ and $\nabla_x\mathcal{N}(t,s)$ are related
to some logarithmic Sobolev inequalities with respect to the tight system $\{\mu_t:\,\,t\in I\}$. These estimates are the natural counterpart of the Sobolev embedding theorems in the context of invariant measures and evolution systems of measures.

For what concerns the ultraboundedness of $\mathcal{N}(t,s)$ and $\nabla_x \mathcal{N}(t,s)$ we first prove an Harnack type estimate which establishes a pointwise estimate of  $|{\mathcal N}(t,s)f|^p$ in terms of $G(t,s)|f|^p$ for any $f\in C_b(\Rd)$, $p>p_0$ and $t>s$. This estimate, together with the evolution law and the ultraboundedness of $G(t,s)$ allow us to conclude that, for any $f\in L^p(\Rd,\mu_s)$
and any $t>s$, the function ${\mathcal N}(t,s)f$ belongs to $W^{1,\infty}(\Rd,\mu_t)$ and to prove an estimate of $\|{\mathcal N}(t,s)f\|_{W^{1,\infty}(\Rd,\mu_t)}$ in terms of $\|f\|_{L^p(\Rd,\mu_s)}$.

Finally, assuming that $\psi(t,x,0,0)= 0$ for every $t\in (s,+\infty)$ and $x \in \Rd$, we prove that the trivial solution to the Cauchy problem \eqref{n-sm-pb} is exponentially stable both in $W^{1,p}(\Rd, \mu_t)$ and in $C^1_b(\Rd)$. This means that $\|u_f(t,\cdot)\|_X\le C_Xe^{-\omega_Xt}$ as $t\to +\infty$ for some constants $C_X>0$ and $\omega_X<0$, both when $X=W^{1,p}(\Rd,\mu_t)$
and $X=C^1_b(\Rd)$. In the first case, the space $X$ depends itself on $t$. We stress that, under sufficient conditions on the coefficients of the operators ${\mathcal A}(t)$, which include their convergence
at infinity, in \cite{AngLor10Com, LorLunSch16Str} it has been proved that the measure $\mu_t$ weakly$^*$ converges to a measure $\mu$, which turns out the invariant measure
of the operator $\A_{\infty}$, whose coefficients are the limit as $t\to +\infty$ of the coefficients of the operator $\A(t)$. This gives more information on the convergence to zero of $\|u_f(t,\cdot)\|_{W^{1,p}(\Rd,\mu_t)}$
at infinity. We refer the reader also to \cite{LorLunZam10Asy} for the case of $T$-time periodic coefficients.

To get the exponential stability of the trivial solution in $C_b(\Rd)$, differently from \cite{AL} where a nonautonomous version of the principle of linearized stability is used and more restrictive assumptions on $\psi$ are required, we let $p$ tend to $+\infty$ in the decay estimate of $\|u_f(t, \cdot)\|_{W^{1,p}(\Rd, \mu_t)}$, since all the constants appearing in this estimate admit finite limit as $p$
tends to $+\infty$. In particular, we stress that we do not need any additional assumptions on the differentiability of $\psi$ but, on the other hand, we require that the mild solution $u_f$ of \eqref{n-sm-pb} is actually classical.

\subsection*{Notations}

For $k\ge 0$, by $C^k_b(\Rd)$ we mean the space of the functions  in $C^k(\Rd)$ which are bounded
together with all their derivatives up to the $[k]$-th order.
$C^k_b(\Rd)$ is endowed  with the norm $\|f\|_{C_b^k(\Rd)}=\sum_{|\alpha|\le [k]}\|D^\alpha f\|_{\infty}+\sum_{|\alpha|=[k]}[D^\alpha f]_{C_b^{k-[k]}(\Rd)}$ where $[k]$ denotes the integer part of $k$.
When $k\notin \N$, we use the subscript ``loc'' to denote the space of all $f\in C^{[k]}(\Rd)$ such that the derivatives of order $[k]$ are $(k-[k]$)-H\"older continuous in any compact subset of $\Rd$.
Given an interval $J$, we denote by $B(J\times\Rd;{\rm Lip}(\R^{d+1}))$ and
$C^{\alpha/2,\alpha}(J\times \Rd)$ ($\alpha\in (0,1)$), respectively, the set of
all functions $f:J\times\Rd\times\R\times\Rd\to\R$ such that $f(t,x,\cdot,\cdot)$ is Lipschitz continuous in $\R^{d+1}$, uniformly with respect to $(t,x)\in J\times\Rd$, and the usual parabolic H\"older space.
The subscript ``loc'' has the same meaning as above.

We use the symbols $D_tf$, $D_i f$ and $D_{ij}f$ to denote respectively the time derivative $\frac{\partial f}{\partial t}$ and the spatial derivatives $\frac{\partial f}{\partial x_i}$ and $\frac{\partial^2f}{\partial x_i\partial x_j}$ for any $i,j=1, \ldots,d$.

The open ball in $\Rd$ centered at $ 0$ with radius $r>0$ and its closure are denoted by  $B_r$ and $\overline{B}_r$, respectively.
For any measurable set $A$, contained in $\R$ or in $\Rd$, we denote by $\one_A$ the characteristic function of $A$.
Finally, we write $A\Subset B$ when $A$ is compactly contained in $B$.

\section{Assumptions and preliminary results}
\label{sect-2}
Let $\{\mathcal{A}(t): t\in I\}$ be the family of linear second-order differential
operators defined by \eqref{oper}.
Our standing assumptions on the coefficients of the operators $\A(t)$ are listed here below.

\begin{hyp}\label{base}
\begin{enumerate}[\rm (i)]
\item
The coefficients $q_{ij}, b_{i}$ belong to $C^{\alpha/2,1+\alpha}_{\rm loc}(I\times \Rd)$ for any $i,j=1,\dots,d$ and
some $\alpha \in (0,1)$;
\item
for every $(t,x)\in I \times \Rd$, the matrix $Q(t,x)=[q_{ij}(t,x)]_{ij}$
is symmetric and there exists a function $\kappa:I\times\Rd\to\R$, with positive infimum $\kappa_0$, such that
$\langle Q(t,x)\xi,\xi\rangle\ge \kappa(t,x)|\xi|^2$ for any $(t,x)\in I\times\Rd$ and any $\xi\in\Rd$;
\item
there exists a non-negative function $\varphi\in C^2(\R^d)$, diverging to $+\infty$ as $|x|\to +\infty$, such that
$(\mathcal{A}(t)\varphi)(x)\leq a-c\,\varphi(x)$ for any $(t,x)\in I\times \Rd$ and
some positive constants $a$ and $c$;
\item
there exists a locally bounded function $\rho:I\to\R^+$ such that
$|\nabla_x q_{ij}(t,x)|\le \rho(t)\kappa(t,x)$ for any $(t,x) \in I\times \Rd$ and
 any $i,j=1, \dots,d$, where $\kappa$ is defined in $(ii)$;
\item there exists a function $r:I\times\Rd\to\R$ such that
$\langle\nabla_x b(t,x)\xi,\xi\rangle \le r(t,x) |\xi|^2$ for any $\xi \in \Rd$ and $(t,x) \in I\times \Rd$. Further,
there exists $p_0\in (1,2]$ such that
\begin{equation}
+\infty>\sigma_{p_0}=\sup_{(t,x)\in I\times\Rd}\left (r(t,x)+\frac{d^3(\rho(t))^2\kappa(t,x)}{4\min\{p_0-1,1\}}\right ).
\label{sigma-p0}
\end{equation}
\end{enumerate}
\end{hyp}

\noindent
Under Hypotheses \ref{base}(i)-(iii) (actually even under weaker assumptions) it is possible to associate an evolution operator $\{G(t,s):\, t\geq s\in I\}$ to the operator $\mathcal{A}(t)$ in $C_b(\Rd)$, as
 described in the Introduction.
The function $G(\cdot,\cdot)f$ is continuous in $\{(s,t,x)\in I\times I\times\Rd: s\le t\}$ and
\begin{equation}
(G(t,s)f)(x)=\int_{\Rd}f(y)p(t,s,x,dy),\qquad\;\, I\ni s<t,\;\, x\in\Rd,
\label{repres-formula}
\end{equation}
where $p(t,s,x,dy)$ are probability measures for any $I\ni s<t$, $x,y\in\Rd$.
This implies that $|G(t,s)f|^p\le G(t,s)(|f|^p)$ for any $I\ni s<t$, $f\in C_b(\Rd)$ and $p\ge 1$.
Moreover, Hypotheses \ref{base}(iv) and (v) yield the pointwise gradient estimates
\begin{equation}\label{point_p_D11}
|(\nabla_x G(t,s)f)(x)|^p \le e^{p\sigma_p(t-s)}(G(t,s)|\nabla f|^p)(x),\qquad\;\, f\in C^1_b(\Rd),
\end{equation}
\begin{align}
|(\nabla_x G(t,s)f)(x)|^p \le \left\{
\begin{array}{ll}
C_{\varepsilon}^pe^{p(\sigma_p+\varepsilon)(t-s)}(t-s)^{-\frac{p}{2}}(G(t,s)|f|^p)(x),
& {\rm if}~\sigma_p<0,\\[2mm]
C_0^p(1+(t-s)^{-\frac{p}{2}})(G(t,s)|f|^p)(x), & {\rm otherwise},
\end{array}
\right.
\label{point_p_D}
\end{align}
for any $f\in C_b(\Rd)$, $t>s$, $x\in \Rd$, $p \in [p_0,+\infty)$, $\varepsilon>0$ and
some positive constants $C_0$ and $C_{\varepsilon}$,
where $\sigma_p$ is given by \eqref{sigma-p0}, with $p$ instead of $p_0$.
We stress that the pointwise estimates \eqref{point_p_D11} and \eqref{point_p_D} have been proved
with the constants $C_0$ and $C_{\varepsilon}$ also depending of $p$. Actually, these constants may be taken independent of
$p$. Indeed, consider for instance estimate \eqref{point_p_D}. If $p\ge p_0$, then using the representation formula \eqref{repres-formula} we can estimate
\begin{align*}
|\nabla_xG(t,s)f|^p=&(|\nabla_xG(t,s)|^{p_0})^{\frac{p}{p_0}}\le (C_{p_0}^{p_0}(1+(t-s)^{-\frac{p_0}{2}})G(t,s)|f|^{p_0})^{\frac{p}{p_0}}\\
\le &2^{\frac{p-p_0}{p_0}}C_{p_0}^p(1+(t-s)^{-\frac{p}{2}})(G(t,s)|f|^{p_0})^{\frac{p}{p_0}}\\
\le &2^{\frac{p-p_0}{p_0}}C_{p_0}^p(1+(t-s)^{-\frac{p}{2}})G(t,s)|f|^p
\end{align*}
for any $t>s\in I$ and $f\in C_b(\Rd)$, and, hence, estimate \eqref{point_p_D} holds true with a constant which can be taken independent of $p$.

\begin{rmk}
\label{rmk-2.4}{\rm
If the diffusion coefficients are bounded and independent of $x$, then the pointwise gradient estimate \eqref{point_p_D11} holds true also with $p=1$
and $\sigma_1=r_0$, where $r_0$ is the supremum over $I\times\Rd$ of the function
$r$ in Hypothesis \ref{base}(v).}
\end{rmk}

Under Hypotheses \ref{base} we can also associate an evolution system of measures $\{\mu_t: t\in I\}$ with the operators ${\mathcal A}(t)$.
Such a family of measures is tight, namely
for every $\varepsilon>0$ there exists $r>0$ such that $\mu_s(\Rd\setminus B_r)<\varepsilon$ for any $s \in I$.
The invariance property \eqref{inv_pro} and the density of $C_b(\Rd)$ in $L^p(\Rd, \mu_s)$, $s \in I$, allows to extend
$G(t,s)$ to a contraction from $L^p(\Rd,\mu_s)$ to $L^p(\Rd,\mu_t)$ for any $t>s$.
As it has been stressed in the Introduction, in general evolution systems of measures are infinitely many, but, under suitable assumptions, there exists a unique
tight evolution system of measures. This is, for instance, the case when Hypotheses \ref{base} are satisfied as
well as the following two conditions:
\begin{enumerate}
\item
$q_{ij}$, $b_i$ belong to $C^{\alpha/2,1+\alpha}_{\rm loc}([a,+\infty)\times \R^d)$ for any $i,j=1,\ldots,d$ and some $a\in I$. Moreover,
$q_{ij}\in C_b([a,+\infty)\times B_R)$ and
$D_kq_{ij},b_j\in C_b([a,+\infty);L^p(B_R))$ for any $i,j,k\in\{1,\ldots,d\}$, $R>0$ and some $p>d+2$;
\item
there exists a constant $c>0$ such that either $|Q(t,x)|\le c(1+|x|)\varphi(x)$ and
$\langle b(t,x),x\rangle \le c(1+|x|^2)\varphi(x)$
for any $(t,x)\in [a,+\infty)\times \R^d$, or the diffusion coefficients are bounded in $[a,+\infty)\times\Rd$.
\end{enumerate}

For more details and the proofs of the results that we have mentioned, we refer the reader to \cite{KunLorLun09Non, LorLibro-2, LorLunSch16Str,LorZam}.

\section{The semilinear problem in a bounded time interval}\label{Sp}
Given $I\ni s<T$, we are interested in studying the Cauchy problem \eqref{n-sm-pb}
both in the case when $f\in C_b(\Rd)$ and in the case when $f\in L^p(\Rd,\mu_s)$.

Let us introduce the standing assumptions on $\psi$.

\begin{hyp}\label{hyp-1}
\begin{enumerate}[\rm (i)]
\item
The function $\psi:{[s,T]}\times \Rd\times\R\times\Rd\to \R$ is continuous.
Moreover, there exist $\beta \in [0,1)$ such that for any $R>0$ and some constant $L_R>0$
\begin{align}
&|\psi(t,x,u_1,(t-s)^{-1/2}v_1)-\psi(t,x,u_2,(t-s)^{-1/2}v_2)|\notag\\
\le &L_R(t-s)^{-\beta}(|u_1-u_2|+|v_1-v_2|),
\label{loc_Lip}
\end{align}
for any $t\in(s,T]$, $x \in \Rd$, $u_1,u_2\in [-R,R]$, $v_1, v_2\in \overline{B}_R$;
\item
the function $\psi(\cdot, \cdot,0,0)$ belongs to $C_b([s,T]\times\Rd)$.
\end{enumerate}
\end{hyp}

\begin{thm}\label{exi_cb}
Under Hypotheses $\ref{base}$ and $\ref{hyp-1}$, for any 
$\overline{f}\in C_b(\Rd)$ there exist $r_0, \delta\in (0,T-s]$ such that, if $f \in C_b(\Rd)$ and $\|f-\overline{f}\|_{\infty}\le r_0$, then
problem \eqref{n-sm-pb} admits a unique mild solution $u_f \in C_b([s,s+\delta]\times \Rd)\cap C^{0,1}((s, s+\delta]\times \Rd)$ of \eqref{n-sm-pb} which satisfies the estimate
\begin{align}
&\|u_f\|_{\infty}+\sup_{t\in(s,s+\delta]}\sqrt{t-s}\|\nabla_x u_f(t,\cdot)\|_\infty\notag\\
\le &2(1+2C_0+C_0\sqrt{\delta})(\|f\|_{C_b(\Rd)}+2\delta\|\psi(\cdot,\cdot,0,0)\|_{C_b([s,s+\delta]\times\Rd)}),
\label{st-gr}
\end{align}
where $C_0$ is the constant in \eqref{point_p_D}.
Moreover, for any $R>0$, $\theta\in (0,1)$ and $t\in (s,s+\delta]$, $u_f(t,\cdot)$ belongs to $C^{1+\theta}(B_R)$ and there exists a positive
constant $C_{R,T-s}$ such that $\sup_{t\in (s,s+\delta]}(t-s)^{(1+\theta)/2}\|u_f(t,\cdot)\|_{C^{1+\theta}(B_R)}\le C_{R,T-s}\|f\|_{\infty}$. Finally, if
$g \in C_b(\Rd)$ is such that $\|g-\overline{f}\|_{\infty}\le r_0$, then
\begin{equation}\label{dip_dati}
\|u_f-u_g\|_{\infty}+\sup_{t\in(s,s+\delta]}\sqrt{t-s}\|\nabla_x u_f(t,\cdot)-\nabla_x u_g(t,\cdot)\|_\infty
\le 2(1+C_0+C_0\sqrt{\delta})  \|f-g\|_{\infty}.
\end{equation}
\end{thm}

\begin{proof}
Even if the proof is quite standard, for the reader's convenience we provide some details.

Fix $\overline{f}\in C_b(\Rd)$ and let $R_0>0$ be such that $R_0/(1+K_0) \ge 8\|\overline{f}\|_{\infty}$,
where $K_0=C_0(1+\sqrt{T-s})$ and $C_0$ is the constant in \eqref{point_p_D}.
Further, for any $\delta\in (0,T-s]$, let $Y_{\delta}$ be the set of all $u\in  C_b([s,s+\delta]\times \Rd)\cap C^{0,1}((s,s+\delta)\times \Rd)$ such that
$\|u\|_{Y_{\delta}}=\|u\|_{C_b((s,s+\delta]\times \Rd)}+ \sup_{t\in (s,s+\delta]}\sqrt{t-s}\|\nabla_xu(t,\cdot)\|_\infty<+\infty$.

{\em Step 1.} Here, we prove that there exists $\delta>0$ such that, for any $f\in C_b(\Rd)$ satisfying the condition $\|f-\overline f\|_{\infty}\le r_0:=R_0/(4+4K_0)$,
there exists a mild solution to problem \eqref{n-sm-pb} defined in the time interval $[s,s+\delta]$. For this purpose, we consider the operator $\Gamma$, defined by the right-hand side of \eqref{defi_mild}
for any $u \in B_{Y_{\delta}}(R_0)$ (the ball of $Y_{\delta}$ centered at zero with radius $R_0$).
Clearly, the function $\psi_u$ is continuous in $(s,s+\delta]\times \Rd$ and $\psi_u(t,\cdot)$ is bounded in $\Rd$ for any $t \in (s,s+\delta]$. Moreover, estimating
$|\psi_u(t,x)|\le |\psi_u(t,x)-\psi(t,x,0,0)|+|\psi(t,x,0,0)|$ and taking
\eqref{loc_Lip} into account, we can easily show that the function $t\mapsto (t-s)^{\beta}\|\psi_u(t,\cdot)\|_\infty$ is bounded in $(s,s+\delta)$.
Hence, Proposition \ref{smoth_v} and estimates \eqref{norm_g} and \eqref{point_p_D} show that $\Gamma(u)\in Y_{\delta}$ for any $t \in (s,s+\delta]$
and $u\in B_{Y_{\delta}}(R_0)$.
To show that, for a suitable $\delta\in (0,1]$, $\Gamma$ is a $1/2$-contraction in $B_{Y_{\delta}}(R_0)$, we observe that,
using again \eqref{loc_Lip}, it follows that
\begin{align}
\|\psi_u(t,\cdot)- \psi_v(t, \cdot)\|_\infty
\le L_{R_0}(t-s)^{-\beta}\|u-v\|_{Y_{\delta}},\qquad\;\,t\in (s,s+\delta],
\label{stima-psi-neu}
\end{align}
for any $u, v \in B_{Y_{\delta}}(R_0)$, where $L_{R_0}$ is the constant in Hypothesis \ref{hyp-1}(i).
From this inequality and estimates \eqref{norm_g} and \eqref{point_p_D} we conclude that
$\|\Gamma(u)-\Gamma(v)\|_{Y_{\delta}}\le c_1\delta^{1-\beta}\|u-v\|_{Y_{\delta}}$ for any $u,v\in B_{Y_{\delta}}(R_0)$, where
$c_1$, as the forthcoming constants, is independent of $\delta$ and $u$, if not otherwise specified.
Hence, choosing $\delta$ properly, we can make $\Gamma$ a $1/2$-contraction in $B_{Y_{\delta}}(R_0)$.

It is also straightforward to see that $\Gamma$ maps $B_{Y_{\delta}}(R_0)$ into itself, up to replacing $\delta$ with a smaller value if needed. It suffices to split $\Gamma(u)=(\Gamma(u)-\Gamma(0))+\Gamma(0)$, use the previous result and estimate
$\|\Gamma(0)\|_{Y_{\delta}}\le(1+C_0+C_0\sqrt{\delta})\|f\|_{\infty}+\delta(1+2C_0+C_0\sqrt{\delta})\|\psi(\cdot,\cdot,0,0)\|_{C_b([s,T]\times\Rd)}$.
As a consequence, $\Gamma$ has a unique fixed point in $B_{Y_{\delta}}(R_0)$, which is a mild solution of \eqref{n-sm-pb} and satisfies \eqref{st-gr}.

{\em Step 2.} Here we prove the uniqueness of the mild solution $u_f$. For this purpose, let $u_1, u_2\in Y_{\delta}$ be two mild solutions. By Lemma \ref{lemm-brigida}, the function
$r \mapsto h(r):= \|u_1(r,\cdot)-u_2(r,\cdot)\|_\infty+ \sqrt{r-s}\|\nabla_x u_1(r,\cdot)-\nabla_x u_2(r,\cdot)\|_\infty$ is measurable in $(s,s+\delta)$.
Moreover, using \eqref{stima-psi-neu}, we easily deduce that
\begin{align}\label{diff-fun}
\|D_x^ju_1(t,\cdot)-D_x^ju_2(t,\cdot)\|_\infty
\le c_2(M)\int_s^t (t-r)^{-\frac{j}{2}}(r-s)^{-\beta}h(r)dr,
\end{align}
for $j=0,1$, any $t\in [s,s+\delta]$, where $M=\max\{\|u_1\|_{Y_{\delta}},\|u_2\|_{Y_{\delta}}\}$.
Estimating $\sqrt{t-s}$ with $\sqrt{t-r}+\sqrt{r-s}$ for any $r \in (s,t)$, from \eqref{diff-fun}, with $j=1$, it follows that
\begin{align}
&\sqrt{t-s}\|\nabla_x u_1(t,\cdot)-\nabla_x u_2(t,\cdot)\|_\infty \notag\\
\le &c_2(M)\int_s^t(r-s)^{-\beta}h(r)dr+c_2(M)\int_s^t(t-r)^{-\frac{1}{2}}(r-s)^{\frac{1}{2}-\beta}\|u_1(r,\cdot)\!-\!u_2(r,\cdot)\|_{\infty}dr\notag\\
&+c_2(M)\int_s^t(t-r)^{-\frac{1}{2}}(r-s)^{1-\beta}\|\nabla_xu_1(r,\cdot)\!-\!\nabla_xu_2(r,\cdot)\|_{\infty}dr.
\label{ale}
\end{align}
Using \eqref{diff-fun} we estimate the last two integral terms in the right-hand side of \eqref{ale}, which we denote by
${\mathcal I}(t)$ and ${\mathcal J}(t)$.
Replacing \eqref{diff-fun}, with $j=0$, in ${\mathcal I}(t)$, we get
\begin{align}
{\mathcal I}(t)&\le c_3(M)\delta^{1-\beta}\int_s^t(\sigma-s)^{-\beta}h(\sigma)d\sigma.
\label{A}
\end{align}

The same arguments show that ${\mathcal J}(t)$ can be estimated pointwise in $[s,s+\delta]$ by the right-hand side of \eqref{A}, with $c_3(M)$ being
possibly replaced by a larger constant $c_4(M)$.
Summing up, we have proved that
\begin{align}
\sqrt{t-s}\|\nabla_x u_1(t,\cdot)-\nabla_x u_2(t,\cdot)\|_\infty\le c_5(M)\delta^{1-\beta}\int_s^t(\sigma-s)^{-\beta}h(\sigma)d\sigma.
\label{ale-1}
\end{align}
From \eqref{diff-fun} and \eqref{ale-1} we conclude that
\begin{eqnarray*}
h(t)\le c_6(M,\delta)\int_s^t (r-s)^{-\beta}h(r)dr,\qquad\;\,t\in (s,s+\delta].
\end{eqnarray*}
The generalized Gronwall lemma (see \cite{Gronwall}) yields that $h(t)\equiv 0$ for any $t \in (s,s+\delta)$, i.e., $u_1\equiv u_2$ in $(s,s+\delta)\times \Rd$.

{\em Step 3.} Here, we prove \eqref{st-gr} and \eqref{dip_dati}.
Since $u_f=\Gamma(0)+(\Gamma(u_f)-\Gamma(0))$ and $\Gamma$ is a $1/2$-contraction in $B_{Y_{\delta}}(R_0)$, we conclude that
$\|u_f\|_{Y_{\delta}}\le 2\|\Gamma(0)\|_{Y_{\delta}}$ and \eqref{st-gr} follows from the estimate on $\|\Gamma(0)\|_{Y_{\delta}}$ proved above.
Estimate \eqref{dip_dati} can be proved in the same way.

{\em Step 4.} Here, we prove that $u_f(t,\cdot)\in C^{1+\theta}(B_R)$ for any $t\in (s,s+\delta]$, $R>0$, $\theta\in (0,1)$, and $\sup_{t\in (s,s+\delta]}(t-s)^{(1+\theta)/2}\|u_f(t,\cdot)\|_{C^{1+\theta}(B_R)}\le
c_7\|f\|_{\infty}$ for some constant $c_7$, independent of $f$.
For this purpose, we observe that the results in the previous steps show that the function $\psi_u$ satisfies the estimate
$(t-s)^{\beta}\|\psi_u(t,\cdot)\|_\infty\leq c_8\|f\|_\infty$ for any $t\in (s,s+\delta]$, the constant $c_8$ being independent of $f$. Applying
Proposition \ref{smoth_v} and estimate \eqref{28AL} we complete the proof.
\end{proof}

\begin{coro}
\label{coro-classica}
In addition to the assumption of Theorem $\ref{exi_cb}$ suppose that
there exist $\beta\in [0,1)$ and $\gamma\in(0,1)$ such that $2\beta+\gamma<2$ and
\begin{align}
|\psi(t,x,u,(t-s)^{-1/2}v)-\psi(t,y,u,(t-s)^{-1/2}v)|
\leq C_R(t-s)^{-\beta}|x-y|^\gamma,
\label{hyp-holder}
\end{align}
for any $t\in(s,T]$, $x,y,v\in B_R$, $u\in [-R,R]$, any $R>0$ and some positive constant $C_R$.
Then, for any $f\in C_b(\Rd)$ the mild solution $u_f$ to problem \eqref{n-sm-pb}
belongs to $C^{1,2}((s,s+\delta]\times \Rd)$ and it is a classical solution to \eqref{n-sm-pb}.
\end{coro}

\begin{proof}
Fix $R>0$. Theorem \ref{exi_cb} shows that $u_f(t,\cdot)$ belongs to $C^{1+\gamma}(B_R)$
and $\|\nabla_xu_f(t,\cdot)\|_{C^{\gamma}(B_R)}\le C_R(t-s)^{-(1+\gamma)/2}\|f\|_{\infty}$ for any $t\in (s,s+\delta]$.
Moreover, by interpolation from \eqref{st-gr} it follows that
$\|u_f(t,\cdot)\|_{C^{\gamma}_b(\Rd)}\le C(t-s)^{-\gamma/2}\|f\|_{\infty}$ for any $t\in (s,s+\delta]$.
From these estimates, adding and subtracting $\psi(t,y,u(t,x),\nabla_xu(t,x))$, we deduce that
$|\psi_u(t,x)-\psi_u(t,y)|\leq C\|f\|_\infty (t-s)^{-\beta-\frac{\gamma}{2}}|x-y|^{\gamma}$
for any $t \in (s, s+\delta]$, $x,y\in\Rd$ such that $|x-y|\le R$ and some positive constant $C$, depending on $R$ and $u$.
As a byproduct, $\|\psi_u(t,\cdot)\|_{C^{\gamma}(B_R)}\le \widetilde C(t-s)^{-\beta-\frac{\gamma}{2}}\|f\|_{\infty}$
for any $t \in (s, s+\delta]$ and some positive constant $\widetilde C$, depending on $R$ and $u$.
Now, using Proposition \ref{smoth_v} we conclude that $u\in C^{1,2}((s,s+\delta]\times\Rd)\cap C^{0,2+\theta}_{\rm loc}((s,s+\delta]\times\Rd)$
for any $\theta<\gamma$, if $\gamma\le\alpha$, and for $\theta=\gamma$ otherwise.
\end{proof}

\begin{rmk}
\label{rem-3.5}
{\rm
Suppose that \eqref{loc_Lip} is replaced by the condition
$|\psi(\cdot,\cdot,u_1,v_1)-\psi(\cdot,\cdot,u_2,v_2)|\le L_R(|u_1-u_2|+|v_1-v_2|)$
in $[s,T]\times\Rd$, for any $R>0$, $u_1, u_2\in [-R,R]$, $v_1, v_2\in B_R$ and some positive constant $L_R$.
Then, the proof of the previous theorem can be repeated verbatim with $Y_{\delta}=C^{0,1}_b([s,s+\delta]\times \Rd)$,
endowed with the natural norm, and we can show that the mild solution to problem \eqref{n-sm-pb} belongs to $C^{0,1}_b([s,s+\delta]\times\Rd)$ and
$\|u_f\|_{C^{0,1}_b([s,s+\delta]\times\Rd)}\le \widetilde C_{\delta}\|f\|_{C_b^1(\Rd)}$
for some positive constant $\widetilde C_{\delta}$, independent of $f$.}
\end{rmk}

We now provide some sufficient conditions for the mild solution to problem \eqref{n-sm-pb} to exist in the large. Such conditions will
be crucial to define the nonlinear evolution operator associated with the Cauchy problem \eqref{n-sm-pb}.
We introduce the following additional assumptions.
\begin{hyp}\label{glob2}
\begin{enumerate}[\rm (i)]
\item
For any $R>0$ there exists a positive constant $L_R$ such that
$|\psi(t,x,u_1,v_1)-\psi(t,x,u_2,v_2)|\le L_R(|u_2-u_1|+|v_2-v_1|)$
for any $t\in [s,T]$, $x\in\Rd$, $u_1, u_2\in [-R,R]$ and $v_1, v_2\in\Rd$;
\item
for any $\tau>s\in I$ there exist positive constants $k_0$, $k_1$ and $a$, and a function $\tilde\varphi\in C^2(\R^d)$ with non-negative values and blowing up at infinity such that
$u\psi(t,x,u,v) \leq k_0(1+u^2)+k_1|u||v|$ and $\mathcal{A}\tilde\varphi+k_1|\nabla\tilde\varphi|\leq a\tilde\varphi$ in $\Rd$
for any $t\in [s,\tau]$, $x,v\in\Rd$ and $u\in\R$.
\end{enumerate}
\end{hyp}

In the rest of this section, for any $p\in [p_0,+\infty)$ and $T>s$ we denote by
$[\psi]_{p,T}$ the supremum over $(s,T)$ of the function $\sqrt{t-s}\,\|\psi(t,\cdot,0,0)\|_{L^p(\Rd,\mu_t)}$. $[\psi]_{\infty,T}$ is defined similarly,
 replacing $L^p(\Rd,\mu_t)$ by $C_b(\Rd)$.

\begin{thm}
\label{prop-3.8}
Assume that Hypotheses $\ref{base}$, $\ref{hyp-1}(ii)$, $\ref{glob2}$ and condition \eqref{hyp-holder} are satisfied. Then, for any $f \in C_b(\Rd)$, the classical solution $u_f$
to problem \eqref{n-sm-pb} exists in $[s,T]$. If, further, the constant in Hypothesis $\ref{glob2}(i)$ is independent of $R$, then for any $p\in [p_0,+\infty]$,
\begin{align}
&\sup_{t \in (s,T)}(\|u_f(t,\cdot)\|_{L^p(\Rd,\mu_t)}
+\sqrt{t-s}\,\|\nabla_x u_f (t,\cdot)\|_{L^p(\Rd, \mu_t)})\notag\\
\le &C_{T-s}(\|f\|_{L^p(\Rd, \mu_s)}
+(\sqrt{T-s}+1)[\psi]_{p,T}),
\label{dip-p-weak-0}\\[2mm]
&\sup_{t \in (s,T)}(\|u_f(t,\cdot)-u_g(t,\cdot)\|_{L^p(\Rd,\mu_t)}+\sqrt{t-s}\,\|\nabla_x u_f (t,\cdot)-\nabla_x u_g(t,\cdot)\|_{L^p(\Rd, \mu_t)})\notag\\
\le &C_{T-s}\|f-g\|_{L^p(\Rd, \mu_s)}
\label{dip-p-weak}
\end{align}
for every $f,g \in C_b(\Rd)$, where $C_{\tau}=(\sqrt{\tau}+1)e^{d_1\tau^{3/2}+d_2}$
for some positive constants $d_1$ and $d_2$.
\end{thm}

\begin{proof}
We split the proof into two steps.

{\em Step 1.} Here, we prove that, for any $f\in C_b(\Rd)$, $u_f$ is defined in the whole $[s,T]$.
For this purpose, we fix $f\in C_b(\Rd)$, denote by $[s,\tau_f)$ the maximal time domain where $u_f$ is defined and assume, by contradiction, that
$\tau_f<T$.
 We are going to prove that $u_f$ is bounded in $[s,\tau_f)\times\Rd$. Once this is proved,
we can use Hypotheses \ref{glob2}(i) to deduce, adding and subtracting $\psi(t,x,0,0)$, that
$|\psi(t,x,u_f(t,x),v)|\le C(1+|v|)$
for $t \in [s,T]$, $x,v\in \Rd$ and some positive constant $C$, which depends on $\|u_f\|_{C_b([s,\tau_f)\times \Rd)}$ and $\|\psi(\cdot,\cdot,0,0)\|_{C_b([s,T]\times\Rd)}$.
Applying the same arguments as in Step 2 of the proof of Theorem \ref{exi_cb}, we can show that also the function $t\mapsto \sqrt{t-s}\|\nabla_xu_f(t,\cdot)\|_{\infty}$
is bounded in $[s,\tau_f)\times\Rd$. This is enough to infer that $u_f$ can be extended beyond $\tau_f$, contradicting the maximality of the interval $[s,\tau_f)$.

To prove that $u_f$ is bounded in $(s,\tau_f)\times \Rd$, we fix $b\in(0,\tau_f-s)$, $\lambda>a+k_0$ and we set
$v_n(t,x):=e^{-\lambda(t-s)}u_f(t,x)-n^{-1}\tilde\varphi(x)$ for any $(t,x)\in [s,s+b]\times\Rd$.
A straightforward computation shows that
\begin{align}
D_tv_n-\A v_n
=& e^{-\lambda (\cdot-s)}\psi(\cdot,\cdot,e^{\lambda(\cdot-s)}(v_n+n^{-1}\tilde\varphi),e^{\lambda(\cdot-s)}(\nabla_xv_n+n^{-1}\nabla\tilde\varphi))\notag\\
&-\lambda(v_n+n^{-1}\tilde\varphi)+n^{-1}\A\tilde\varphi,
\label{mercato}
\end{align}
in $(s,s+b]\times\Rd$. Since $u_f$ is bounded in $[s,s+b]\times\Rd$ and $\tilde\varphi$ blows up at infinity,
the function $v_n$ admits a maximum point $(t_n,x_n)$. If $v_n(t_n,x_n)\leq0$ for any $n$,
then $u_f\leq0$ in $[s,s+b]\times\Rd$. Assume that $v_n(t_n,x_n)>0$ for some $n$. If $t_n=s$,
then $v_n(t_n,x_n)\leq \sup_{\Rd}f$. If $t_n>s$ then $D_tv_n(t_n,x_n)-\A(t_n)v_n(t_n,x_n)\geq0$,
so that, multiplying both the sides of \eqref{mercato} by $v_n(t_n,x_n)+n^{-1}\tilde\varphi(x_n)>0$ and using Hypotheses \ref{glob2}(ii) we get
$0\le (-\lambda+k_0+a)(v_n(t_n,x_n)+n^{-1}\tilde\varphi(x_n))^2+k_0$,
which clearly implies that, also in this case, $u$ is bounded from above in $[s,s+b]$ by a constant, independent of $b$.

Repeating the same arguments with $u_f$ being replaced by $-u_f$, we conclude that $u_f$ is bounded also from below by a positive constant
independent of $b$. Since $b$ is arbitrary, it follows that $\|u_f\|_{C_b((s,\tau_f)\times\Rd)}<+\infty$ as claimed.

{\em Step 2.} Fix $f, g\in C_b(\Rd)$ $p\ge p_0$ and let $\opnorm{\cdot}_p$ be the norm defined by
$\opnorm{v}_p=\sup_{t\in (s,T)}e^{-\omega(t-s)}(\|v(t,\cdot)\|_{L^p(\Rd,\mu_t)}+\sqrt{t-s}\|\nabla_xv(t,\cdot)\|_{L^p(\Rd,\mu_t)})$ on smooth functions $v$, where $\omega$
is a positive constant to be chosen later on and to fix the ideas we assume that $p<+\infty$.
From Hypothesis \ref{glob2}(i), where $L_R$ is replaced by a constant $L$, it follows that
$\|\psi_{u_g}(r,\cdot)-\psi_{u_f}(r,\cdot)\|_{L^p(\Rd,\mu_r)}\le L\|u_g(r,\cdot)-u_f(r,\cdot)\|_{W^{1,p}(\Rd,\mu_r)}$
for any $r \in (s,T]$. Hence, recalling that each operator $G(t,r)$ is a contraction from $L^p(\Rd,\mu_r)$ to $L^p(\Rd,\mu_t)$ and using the second pointwise gradient
estimate in \eqref{point_p_D} and the invariance property of the family $\{\mu_t: t\in I\}$,
we conclude that
\begin{align}\label{commissioni}
&\opnorm{u_f-u_g}_p\notag\\
\le & \opnorm{G(\cdot,s)(f-g)}_p+
L\opnorm{u_f-u_g}_p\sup_{t\in (s,T)}\int_s^te^{\omega(r-t)}\bigg (1+\frac{1}{\sqrt{r-s}}\bigg )dr\notag\\
&+LC_0\opnorm{u_f-u_g}_p\sup_{t\in (s,T)}\sqrt{t-s}\int_s^te^{\omega(r-t)}\bigg (1+\frac{1}{\sqrt{t-r}}\bigg )\bigg (1+\frac{1}{\sqrt{r-s}}\bigg )dr\notag\\
\le &[1+C_0(1+\sqrt{T-s})]\|f-g\|_{L^p(\Rd,\mu_s)}\notag\\
&+L\opnorm{u_f-u_g}_p\bigg [(1+C_0\sqrt{T-s})\bigg (\frac{1}{\omega}
+\int_s^t\frac{e^{\omega(r-t)}}{\sqrt{r-s}}dr\bigg )+\frac{\sqrt{\pi}}{\sqrt{\omega}}
C_0\sqrt{T-s}\notag\\
&\qquad\qquad\qquad\quad\;\;\;\,+C_0
\sup_{t\in (s,T)}\sqrt{t-s}\int_s^t\!\frac{e^{\omega(r-t)}dr}{\sqrt{t-r}\sqrt{r-s}}\bigg ].
\end{align}

To estimate the integral terms in the last side of \eqref{commissioni}, we fix $\delta>0$ and observe that
\begin{align}
\int_s^t \frac{e^{\omega(r-t)}}{\sqrt{r-s}}dr
=\int_s^{s+\delta} \frac{e^{\omega(r-t)}}{\sqrt{r-s}}dr+\left (\int_{s+\delta}^t \frac{e^{\omega(r-t)}}{\sqrt{r-s}}dr\right )^+\le
2\sqrt{\delta}+\frac{1}{\sqrt{\delta}\omega}.
\label{acqua-6a}
\end{align}
Hence, minimizing over $\delta>0$ we conclude that the left-hand side of \eqref{acqua-6a} is bounded from above by $\sqrt{8}\omega^{-1/2}$.
Splitting $\sqrt{t-s}\le \sqrt{t-r}+\sqrt{r-s}$ and arguing as above, also the last term in square brackets in the last side of \eqref{commissioni} can be estimated by $(\sqrt{8}+\sqrt{\pi})\omega^{-1/2}$.
It thus follows that
\begin{align*}
\opnorm{u_f-u_g}_p\le &[1+C_0(1+\sqrt{T-s})]\|f-g\|_{L^p(\Rd,\mu_s)}\\
&+L(c_{T-s}\omega^{-1/2}+(1+C_0\sqrt{T-s})\omega^{-1})\opnorm{u_f-u_g}_p,
\end{align*}
where $c_{\tau}=(\sqrt{8}+\sqrt{\pi})C_0(\sqrt{\tau}+1)+\sqrt{8}$.
Choosing $\omega$ such that $c_{T-s}\omega^{-1/2}+(1+C_0\sqrt{T-s})\omega^{-1}\le (2L)^{-1}$, we obtain
\begin{eqnarray*}
\opnorm{u_f-u_g}_p\le 2[1+C_0(1+\sqrt{T-s})]\|f-g\|_{L^p(\Rd,\mu_s)}
\end{eqnarray*}
and estimate \eqref{dip-p-weak} follows at once. Estimate \eqref{dip-p-weak-0} can be proved likewise.
Hence, the details are omitted.
\end{proof}

As a consequence of Theorem \ref{prop-3.8} we prove the existence of a mild solution to problem \eqref{n-sm-pb} in the time domain $(s,T)$ when $f\in L^p(\Rd,\mu_s)$,
that is a function $u_f\in L^p((s,T)\times \Rd,\mu)\cap W^{0,1}_p(J\times \Rd,\mu)$, for any $J\Subset (s,T]$, such that $u_f(t,\cdot)\in W^{1,p}(\Rd,\mu_t)$ for almost every $t\in (s,T]$ and, for such values of $t$, the equality
\begin{align*}
u_f(t,x)= (G(t,s)f)(x)+\int_s^t(G(t,r)\psi_u(r,\cdot))(x)dr,
\end{align*}
holds true in $\Rd\setminus A_t$, where $A_t$ is negligible with respect to the measure $\mu_t$ (or, equivalently, with respect to the restriction of the Lebesgue measure
to the Borel $\sigma$-algebra in $\Rd$).

\begin{coro}
\label{coro-Lp}
Under all the assumptions of Theorem $\ref{prop-3.8}$, for any $f\in L^p(\Rd,\mu_s)$ $(p\ge p_0)$ there exists a unique
mild solution to the Cauchy problem \eqref{n-sm-pb}.
The function $u_f$ satisfies estimates \eqref{dip-p-weak-0} and \eqref{dip-p-weak} with the supremum being replaced by the essential supremum and,
as a byproduct, $u_f\in W^{0,1}_p((s,T)\times\Rd,\mu)$, if $p<2$, and $u_f\in W^{0,1}_q((s,T)\times\Rd,\mu)$ for any $q<2$, otherwise.
Finally, if there exists $\gamma\in(0,1)$ such that
\begin{align}
\label{tazza}
|\psi(t,x,\xi,\eta)-\psi(t,y,\xi,\eta)|\leq C_{J,R}(1+|\xi|+|\eta|)|x-y|^\gamma,
\end{align}
for any $t\in J$, $x,y\in B_R$, $\eta\in \Rd$, $\xi\in\R$, $J\Subset (s,T]$, $R>0$ and a constant $C_{J,R}>0$,
then, for any $f\in L^p(\Rd,\mu_s)$ and almost every $t\in (s,T)$, $u_f(t,\cdot)$ belongs to $W^{2,p}_{\rm loc}(\Rd)$. Moreover, $u_f\in W^{1,2}_{p,\rm loc}((s,T)\times\Rd)$ and satisfies
the equation $D_tu_f=\A u_f+\psi_{u_f}$.
\end{coro}

\begin{proof}
Fix $f\in L^p(\Rd,\mu_s)$ and let $(f_n)\subset C_b(\Rd)$ be a sequence converging to $f$ in $L^p(\Rd,\mu_s)$. By \eqref{dip-p-weak},
$(u_{f_n}(t,\cdot))$ is a Cauchy sequence in $W^{1,p}(\Rd,\mu_t)$ for any $t\in (s,T]$.
Hence, there exists a function $v$ such that $u_{f_n}(t,\cdot)$ converges to $v(t,\cdot)$
in $W^{1,p}(\Rd,\mu_t)$ for any $t\in (s,T]$. Moreover, writing \eqref{dip-p-weak-0}, with $f$ being replaced by $f_n$, and letting $n$
tend to $+\infty$ we deduce that $v$ satisfies \eqref{dip-p-weak-0} as well.

Next, using \eqref{dip-p-weak} we can estimate
\begin{align*}
\|u_{f_n}-u_{f_m}\|_{L^p((s,T)\times\Rd,\mu)}^p=&\int_s^T\|u_{f_n}(t,\cdot)-u_{f_m}(t,\cdot)\|_{L^p(\Rd,\mu_t)}^pdt\\
\le &C_{T-s}^p(T-s)\|f_n-f_m\|_{L^p(\Rd,\mu_s)}^p
\end{align*}
and
\begin{align*}
\|\nabla_xu_{f_n}-\nabla_xu_{f_m}\|_{L^q((s,T)\times\Rd,\mu)}^q=&\int_s^T\|\nabla_xu_{f_n}(t,\cdot)-\nabla_xu_{f_m}(t,\cdot)\|_{L^q(\Rd,\mu_t)}^qdt\\
\le &\frac{2C_{T-s}^q}{2-q}(T-s)^{1-\frac{q}{2}}\|f_n-f_m\|_{L^q(\Rd,\mu_s)}^q
\end{align*}
for any $q\in [1,2)$, if $p\ge 2$ and for $p=q$ otherwise. Hence,
recalling that $L^{q}(\Rd,\mu_t)\hookrightarrow L^p(\Rd,\mu_t)$ for any $t\in I$, we conclude that
the sequence $(u_{f_n})$ converges in $L^p((s,T)\times\Rd,\mu)\cap W^{0,1}_q((s,T)\times\Rd,\mu)$ to a function, which we denote
by $u_f$.
Clearly, $v(t,\cdot)=u_f(t,\cdot)$ almost everywhere in $\Rd$ for almost every $t\in (s,T)$. Letting $n$ tend to $+\infty$ in the formula
\eqref{defi_mild}, with $f_n$ replacing $f$,
we deduce that $u_f$ is a mild solution to problem \eqref{n-sm-pb}. The uniqueness follows, arguing as in the proof of Theorem \ref{exi_cb} with the obvious changes.

Let us now prove the last part of the statement. We again use an approximation argument.
Fix $t>s\in I$, $R>0$. At a first step, we estimate the norm of operator $G(t,r)$ in $\mathcal L(L^p(\Rd,\mu_r),L^p(B_{R+1}))$
and in $\mathcal L(L^p(\Rd,\mu_r),W^{2,p}(B_{R+1}))$, for any $r\in[s,t)$. In the rest of the proof, we denote by $c$ a positive constant, possibly depending on $R$, but
being independent of $t$, $r$ and $f\in L^p(\Rd,\mu_r)$, which may vary from line to line.
Since
there exists a positive and continuous function $\rho:I\times\Rd\to\R$
such that $\mu_r=\rho(r,\cdot)dx$, the spaces $L^p(B_M)$ and $L^p(B_M,\mu_r)$ coincide and their norms are equivalent for any $M>0$. From this remark,
the interior $L^p$-estimates in Theorem \ref{prop-A3}, with $u=G(\cdot,s)f$ and the contractiveness of $G(t,r)$ from $L^p(\Rd,\mu_r)$ to $L^p(\Rd,\mu_t)$,  imply that
\begin{align}
\|G(t,r)f\|_{W^{2,p}(B_{R+1})}
& \leq c(t-r)^{-1}\|f\|_{L^p(\Rd,\mu_r)},\qquad\;\, s<r<t<T,
\label{volante}
\end{align}
first for any $f\in C_b(\Rd)$, and then, by density, for any $f\in L^p(\Rd,\mu_r)$.
Since, for $\theta\in (0,1)$, $(L^p(\Rd,\mu_r), L^p(\Rd,\mu_r))_{\theta,p}=L^p(\Rd,\mu_r)$ and $(W^{1,p}(B_{R+1}), W^{2,p}(B_{R+1}))_{\theta,p}$$=W^{1+\theta,p}(B_{R+1})$,
with equivalence of the corresponding norms, by an interpolation argument and \eqref{volante} we deduce
that $\|G(t,r)\|_{\mathcal L(L^p(\Rd,\mu_r),W^{1+\theta,p}(B_{R+1}))}
\leq c(t-r)^{-\frac{1+\theta}{2}}$ for any $s<r<t<T$.
Hence, if for any $n\in\N$ we consider the function $z_n$, which is the integral term in \eqref{defi_mild}, with $u$ being replaced by $u_{f_n}$,
 and use \eqref{dip-p-weak} and the fact that $\psi\in B([s,T]\times\Rd;{\rm Lip}(\R^{d+1}))$, then we get
\begin{align*}
&\|\nabla_x z_n(t,\cdot)-\nabla_x z_m(t,\cdot)\|_{W^{\theta,p}(B_{R+1})}\\
\leq  & c\int_s^t(t-r)^{-\frac{1+\theta}{2}}
\|u_{f_n}(r,\cdot)-u_{f_m}(r,\cdot)\|_{L^p(\Rd,\mu_r)}dr\\
&+c\int_s^t(t-r)^{-\frac{1+\theta}{2}}\|\nabla_xu_{f_n}(r,\cdot)-\nabla_xu_{f_m}(r,\cdot)\|_{L^p(\Rd,\mu_r)}dr \\
\leq & c(t-s)^{-\frac{\theta}{2}}\|f_n-f_m\|_{L^p(\Rd,\mu_s)}
\end{align*}
for any $n\in\N$. We have so proved that, for any $\theta\in (0,1)$ and almost every $t\in(s,T]$, the function  $u_f(t,\cdot)$ belongs to $W^{1+\theta}(B_{R+1})$
and
\begin{align*}
\|u_{f_n}(t,\cdot)-u_{f_m}(t,\cdot)\|_{W^{1+\theta,p}(B_{R+1})}\leq c(t-s)^{-\frac{1+\theta}{2}}\|f_n-f_m\|_{L^p(\Rd,\mu_s)},\qquad m,n\in\N.
\end{align*}

Similarly,
\begin{align*}
\|u_{f_n}(t,\cdot)\|_{W^{1+\theta,p}(B_{R+1})}\leq c(t-s)^{-\frac{1+\theta}{2}}(\|f\|_{L^p(\Rd,\mu_s)}+\|\psi(\cdot,\cdot,0,0)\|_{\infty}),\qquad\;\,n\in\N.
\end{align*}

Using these estimates, we can now show that $\psi_{u_f}(r,\cdot)\in W^{\theta,p}(B_{R+1})$, for any $\theta<\gamma$. For this purpose, we add and subtract $\psi(t,y,u_{f_n}(t,x),\nabla_xu_{f_n}(t,x))$,
use condition \eqref{tazza} and the Lipschitz continuity of $\psi$ with respect to the last two variables to infer that
\begin{align}
|\psi_{u_{f_n}}(t,x)-\psi_{u_{f_n}}(t,y)|
\le &c|u_{f_n}(t,x)-u_{f_n}(t,y)|+c|\nabla_xu_{f_n}(t,x)-\nabla_xu_{f_n}(t,y)|\notag\\
&+c|x-y|^{\gamma}(1+|u_{f_n}(t,x)|+|\nabla_xu_{f_n}(t,x)|)
\label{label-A}
\end{align}
and
\begin{align}
|\psi_{u_{f_n}}(t,x)-\psi_{u_{f_m}}(t,x)|\le &c(|u_{f_n}(t,x)-u_{f_m}(t,x)|+|\nabla_xu_{f_n}(t,x)-\nabla_xu_{f_m}(t,x)|)
\label{label-B}
\end{align}
for any $t\in (s,T)$, $x,y\in\Rd$ and $m,n\in\N$. Hence, using \eqref{label-A} we obtain
\begin{align*}
&|\psi_{u_{f_n}}(t,x)-\psi_{u_{f_n}}(t,y)-\psi_{u_{f_m}}(t,x)+\psi_{u_{f_m}}(t,y)|\notag\\
\le & c\big [|u_{f_n}(t,x)-u_{f_n}(t,y)|+|\nabla_xu_{f_n}(t,x)-\nabla_xu_{f_n}(t,y)|\notag\\
&\quad\! +|u_{f_m}(t,x)-u_{f_m}(t,y)|+|\nabla_xu_{f_m}(t,x)-\nabla_xu_{f_m}(t,y)|\notag\\
&\quad\! +|x\!-\!y|^{\gamma}(1\!+\!|u_{f_n}(t,x)|\!+\!|u_{f_m}(t,x)|\!+\!|\nabla_xu_{f_n}(t,x)|\!+\!|\nabla_xu_{f_m}(t,x)|)\big ]\!=:\!{\mathcal I}(t,x,y)
\end{align*}
and, using \eqref{label-B},
\begin{align*}
&|\psi_{u_{f_n}}(t,x)-\psi_{u_{f_n}}(t,y)-\psi_{u_{f_m}}(t,x)+\psi_{u_{f_m}}(t,y)|\\
\le &c\big [|u_{f_n}(t,x)-u_{f_m}(t,x)|+|\nabla_xu_{f_n}(t,x)-\nabla_xu_{f_m}(t,x)|\\
&\quad +|u_{f_n}(t,y)-u_{f_m}(t,y)|+|\nabla_xu_{f_n}(t,y)-\nabla_xu_{f_m}(t,y)|\big ]=:{\mathcal J}(t,x,y).
\end{align*}
From these two estimates we conclude that
\begin{align*}
|\psi_{u_{f_n}}(t,x)-\psi_{u_{f_n}}(t,y)-\psi_{u_{f_m}}(t,x)+\psi_{u_{f_m}}(t,y)|^p\le ({\mathcal I}(t,x,y))^{\beta p}({\mathcal J}(t,x,y))^{(1-\beta)p}
\end{align*}
for any $(t,x)\in (s,T)\times\Rd$, any $\beta\in (0,1)$ and any $m,n\in\N$.
Hence, for any $\theta<\gamma$ and $\beta$, such that $(0,1)\ni\theta'=\theta/\beta+d(1-\beta)/(p\beta)$, a long but straightforward computation reveals that
\begin{align*}
&[\psi_{u_{f_n}}(t,\cdot)-\psi_{u_{f_m}}(t,\cdot)]_{W^{\theta,p}(B_{R+1})}\\
\le &c\|u_{f_n}(t,\cdot)-u_{f_m}(t,\cdot)\|_{W^{1,p}(\Rd,\mu_t)}^{(1-\beta)}\notag\\
&\qquad\times\big (
\|u_{f_n}(t,\cdot)\|_{W^{1+\theta',p}(B_{R+1})}^{\beta}+
\|u_{f_m}(t,\cdot)\|_{W^{1+\theta',p}(B_{R+1})}^{\beta}
+1\big )
\end{align*}
and, consequently,
\begin{align*}
\|\psi_{u_{f_n}}(t,\cdot)-\psi_{u_{f_m}}(t,\cdot)\|_{W^{\theta,p}(B_{R+1})}\le c(t-s)^{\frac{\beta-\theta'}{2}-1}\|f_n-f_m\|_{L^p(\Rd,\mu_s)}^{1-\beta}
\end{align*}
for any $t\in (s,T)$.
We are almost done. Indeed, by interpolation from Proposition \ref{prop-A3} we deduce that
$\|G(t,r)\|_{\mathcal L(W^{\theta,p}(B_{R+1}), W^{2,p}(B_R))}\le c(t-r)^{-1+\theta/2}$.
From this and the previous estimate we conclude that
\begin{eqnarray*}
\|z_n(t,\cdot)-z_m(t,\cdot)\|_{W^{2,p}(B_R)}\le c(t-s)^{\frac{\beta+\theta-\theta'}{2}-1}\|f_n-f_m\|_{L^p(\Rd,\mu_s)}^{1-\beta},\qquad\;\,m,n\in\N,
\end{eqnarray*}
for any $t\in (s,T]$ and $\beta>\theta'$, so that $\|u_{f_n}(t,\cdot)-u_{f_m}(t,\cdot)\|_{W^{2,p}(B_R)}\le c(t-s)^{-1}\|f_n-f_m\|_{L^p(\Rd,\mu_s)}^{1-\beta}$
for any $m,n\in\N$, thanks to \eqref{volante}. From this estimate it is easy to conclude that $(u_{f_n})$ is a Cauchy sequence in $W^{0,2}_{p,\rm loc}((s,T)\times\Rd)$. Since $u_{f_n}$ is a classical solution to problem \eqref{n-sm-pb}, $(D_tu_{f_n})$ is a Cauchy sequence in $L^p_{\rm loc}((s,T)\times\Rd)$.
It thus follows that $u_f\in W^{1,2}_{p,\rm loc}((s,T)\times\Rd)$ and it solves the equation $D_tu_f=\A u_{f}+\psi_{u_f}$ in $(s,T)\times\Rd$.
\end{proof}

The arguments in the proof of Theorem \ref{prop-3.8} and Corollary \ref{coro-Lp} allow us to prove the following result.

\begin{prop}
Under Hypotheses $\ref{base}$, the following properties are satisfied.
\begin{enumerate}[\rm (i)]
\item
Let  $\psi\!\in\! C((s,T]\times\Rd\times\R\times\Rd)$ with $[\psi]_{\infty,T}+\sup_{(t,x)\in (s,T]\times\Rd}[\psi(t,x,\cdot,\cdot)]_{{\rm Lip}(\R^{d+1})}$ $<+\infty$.
Then, for any $f\in C_b(\Rd)$, the Cauchy problem \eqref{n-sm-pb} admits a unique mild solution $u_f\in C([s,T]\times\Rd)\cap C^{0,1}((s,T]\times\Rd)$ which
satisfies \eqref{dip-p-weak} and \eqref{dip-p-weak-0} for any $p\in [p_0,+\infty]$.
\item
Let $\psi\in C((s,T]\times\Rd\times\R\times\Rd)$ and $[\psi]_p+\sup_{(t,x)\in (s,T]\times\Rd}[\psi(t,x,\cdot,\cdot)]_{{\rm Lip}(\R^{d+1})}<+\infty$ for some $p\ge p_0$.
Then, for any $f\in L^p(\Rd,\mu_s)$, the Cauchy problem \eqref{n-sm-pb} admits a unique mild solution $u_f$ which belongs to $W^{0,1}_p((s,T)\times\Rd)$, if $p_0\le p<2$, and to $W^{0,1}_q((s,T)\times\Rd)$
for any $J\Subset (s,T]$, if $p\ge 2$. Further, $u_f$ satisfies \eqref{dip-p-weak-0} and \eqref{dip-p-weak}, with the supremum being replaced by the essential supremum.
\end{enumerate}
\end{prop}

\begin{proof}
To prove property (i) it suffices to apply the Banach fixed point theorem in the space of all the functions $v\in C_b([s,T]\times\Rd)\cap C^{0,1}((s,T]\times\Rd)$
such that $\opnorm{v}_{\infty}<+\infty$, where $\opnorm{\cdot}_{\infty}$ is defined in Step 2 of the proof of Theorem \ref{prop-3.8}, with $p=+\infty$. The uniqueness of the so obtained solution follows from the condition $\sup_{(t,x)\in (s,T]\times\Rd}[\psi(t,x,\cdot,\cdot)]_{{\rm Lip}(\R^{d+1})}<+\infty$, in a standard way.

To prove property (ii), one can argue by approximation. We fix $f\in L^p(\Rd,\mu_s)$, approximate it by a sequence $(f_n)\subset C_b(\Rd)$, converging to $f$ in
 $L^p(\Rd,\mu_s)$, and introducing a standard sequence $(\vartheta_n)$ of cut-off functions. If we set $\psi_n=\vartheta_n\psi$ for
any $n\in\N$, then, each function $\psi_n$ satisfies the condition (b) in property (i) and $[\psi_n]_{p,T}\le [\psi]_{p,T}$. Therefore, the Cauchy problem
\eqref{n-sm-pb}, with $f_n$ and $\psi_n$ replacing $f$ and $\psi$ admits a unique mild solution
$u\in C_b([s,T]\times\Rd)\cap C^{0,1}((s,T]\times\Rd)$, which satisfies \eqref{dip-p-weak-0} and \eqref{dip-p-weak} with $f_n$ replacing $f$.
The arguments in the first part of the proof of Corollary \ref{coro-Lp}  allow us to prove the
existence of a mild solution $u_f$ to the Cauchy problem \eqref{n-sm-pb} with the properties in the statement of the proposition.
The uniqueness of the solution
follows also in this case from the condition $\sup_{(t,x)\in (s,T]\times\Rd}[\psi(t,x,\cdot,\cdot)]_{{\rm Lip}(\R^{d+1})}<+\infty$.
\end{proof}

\section{The evolution operator and its summability improving properties}
\label{4}

If, besides Hypotheses \ref{base}, the assumptions on $\psi$ in Theorem \ref{prop-3.8} hold true for any $I\ni s<T$ or if $\psi\in C(I\times\Rd\times\R\times\Rd)\cap B_{\rm loc}(I\times\Rd;{\rm Lip}(\R^{d+1}))$ and $\psi(\cdot,\cdot,0,0)\in C_b(\R^{d+1})$, then
for any $f\in C_b(\Rd)$ and $s\in I$ the mild solution to problem \eqref{n-sm-pb} exists
in the whole of $[s,+\infty)$. Hence, we can set ${\mathcal N}(t,s)f=u_f(t,\cdot)$
for any $t>s$. Each operator ${\mathcal N}(t,s)$ maps $C_b(\Rd)$ into $C^1_b(\Rd)$. Moreover,
the uniqueness of the solution to problem \eqref{n-sm-pb} yields the {\it evolution law}
${\mathcal N}(t,s)f={\mathcal N}(t,r){\mathcal N}(r,s)f$ for any $r\in (s,t)$ and $f\in C_b(\Rd)$. Hence $\{{\mathcal N}(t,s): I\ni s<t\}$ is
a nonlinear evolution operator in $C_b(\Rd)$. It
can be extended to the $L^p$-setting, for any $p\ge p_0$, using the same arguments as in the
first part of the proof of Corollary \ref{coro-Lp}. Clearly, if $\psi(t,x,\cdot,\cdot)$ is Lipschitz continuous in
$\R^{d+1}$, uniformly with respect to $(t,x)\times J\times\Rd$, for any $J\Subset I$, then by density, we still deduce that
${\mathcal N}(t,s)$ satisfies the evolution law and, moreover, each operator
${\mathcal N}(t,s)$ is bounded from $L^p(\Rd,\mu_s)$ to $W^{1,p}(\Rd,\mu_t)$ and
\begin{align}
&\|{\mathcal N}(t,s)f\|_{L^p(\Rd,\mu_t)}
+\sqrt{t-s}\,\|\nabla_x {\mathcal N}(t,s)f\|_{L^p(\Rd, \mu_t)})\notag\\
\le &C_{T-s}(\|f\|_{L^p(\Rd, \mu_s)}+(T-s+1)\|\psi(\cdot,\cdot,0,0)\|_{\infty}).
\label{ciclismo}
\end{align}

\subsection{Continuity properties of the nonlinear evolution operator}

In the following theorem, assuming the above conditions on $\psi$, we prove an interesting continuity property of the operator
${\mathcal N}(t,s)$.

\begin{thm}
\label{compleanno}
Let $(f_n)\subset C_b(\Rd)$ be a bounded sequence converging to some function $f\in C_b(\Rd)$
pointwise in $\Rd$. Then, for any $s\in I$, ${\mathcal N}(\cdot,s)f_n$ and $\nabla_x{\mathcal N}(\cdot,s)f_n$
converge to  ${\mathcal N}(\cdot,s)f$ and $\nabla_x{\mathcal N}(\cdot,s)f$, respectively,
locally uniformly in $(s,+\infty)\times\Rd$.
\end{thm}

\begin{proof}
Let $(f_n)$ and $f$ be as in the statement. To ease the notation, we write
$u_{f_n}$ and $u_f$ for ${\mathcal N}(\cdot,s)f_n$ and ${\mathcal N}(\cdot,s)f$,
respectively. Moreover, for any $n\in\N$, $t>s$ and $r\in (s,t]$ we set
$h_n(r,\cdot)=G(t,r)(|u_{f_n}(r,\cdot)-u_f(r,\cdot)|^p
+|\nabla_x(u_{f_n}(r,\cdot)-u_f(r,\cdot))|^p)$, and we denote by $L_{R,T}$ any constant such that
\begin{equation}
|\psi(t,x,u_2,v_2)-\psi(t,x,u_1,v_1)|\le L_{R,T}(|u_2-u_1|+|v_2-v_1|)
\label{mismo}
\end{equation}
for any $t\in [s,s+T]$, $x,v_1,v_2\in\Rd$, $u_1,u_2\in [-R,R]$ and $T>0$.
As a first step, formula \eqref{st-gr} shows that, for any $T>0$,
there exists a positive constant $M_T$ such
that $\|u_f\|_{\infty}+\|u_{f_n}\|_{\infty}\le M_T$.
Fix $p\in (1,2)$. Using formula \eqref{defi_mild}, we can estimate
\begin{align*}
|D_x^ju_{f_n}(t,x)-D_x^ju_f(t,x)|^p\le &2^{p-1}|(D_x^jG(t,s)(f_n-f))(x)|^p\\
&+2^{p-1}\bigg |\int_s^t(D_x^jG(t,r)(\psi_{u_{f_n}}(r,\cdot)
-\psi_{u_f}(r,\cdot))(x)dr\bigg |^p
\end{align*}
for any $(t,x)\in (s,+\infty)\times\Rd$ and $j=0,1$.
By the representation formula \eqref{repres-formula}, H\"older inequality,
estimates \eqref{point_p_D} and \eqref{mismo}, we deduce that
\begin{align*}
|u_{f_n}(t,\cdot)-u_f(t,\cdot)|^p\le 2^{p-1}G(t,s)|f_n-f|^p+(4T)^{p-1}L_{M_T}^p\int_s^th(r,\cdot)dr
\end{align*}
and
\begin{align*}
|\nabla_xu_{f_n}(t,\cdot)-\nabla_xu_f(t,\cdot)|^p
\le &2^{p-1}(t-s)^{-\frac{p}{2}}c_TG(t,s)|f_n-f|^p\notag\\
&+(4T)^{p-1}c_TL_{M_T,T}^p\int_s^t(t-r)^{-\frac{p}{2}}h(r,\cdot)dr
\end{align*}
in $\Rd$, for any $t\in (s,s+T)$ and some positive constant $c_T$.
Hence, the function $h_n(\cdot,x)$ satisfies the differential inequality
\begin{eqnarray*}
h_n(t,x)\le C_{p,T}(t-s)^{-\frac{p}{2}}(G(t,s)|f_n-f|^p)(x)+C_{p,T}\int_s^t(t-r)^{-\frac{p}{2}}h_n(r,x)dr
\end{eqnarray*}
for any $t\in (s,s+T)$ and $x\in\Rd$. Since $h_n(\cdot,x)$ is continuous in $(s,t]$ and $h_n(r,x)\le
\widetilde C_T(r-s)^{-p/2}$ for some positive constant $\widetilde C_T$, independent of $n$, and any
$r\in (s,t)$, we can apply \cite[Lemma 7.1]{henry} and conclude that
\begin{align*}
h_n(t,x)\le &C_{p,T}(t-s)^{-\frac{p}{2}}(G(t,s)|f_n-f|^p)(x)\\
&+C_{p,T}\int_s^t(t-r)^{-\frac{p}{2}}(r-s)^{-\frac{p}{2}}(G(r,s)|f_n-f|^p)(x)dr
\end{align*}
for any $t\in (s,s+T)$. Hence,
\begin{align*}
\|h_n(t,\cdot)\|_{C_b(B_R)}\le &C_{p,T}(t-s)^{-\frac{p}{2}}\|G(t,s)|f_n-f|^p\|_{C_b(B_R)}\\
&+C_{p,T}\int_s^t(t-r)^{-\frac{p}{2}}(r-s)^{-p/2}\|G(r,s)|f_n-f|^p\|_{C_b(B_R)}dr
\end{align*}
for any $R>0$. By \cite[Proposition 3.1(i)]{KunLorLun09Non}, $\|G(r,s)|f_n-f|^p\|_{C_b(B_R)}$ vanishes
as $n\to +\infty$ for any $r>s$. Hence, by dominated convergence,
$\|h_n(t,\cdot)\|_{C_b(B_R)}$ vanishes as $n\to +\infty$ for any $t\in (s,s+T)$, which means
that, for any $t\in (s,s+T)$, $u_{f_n}(t,\cdot)$ and $\nabla_xu_{f_n}(t,\cdot)$ converge
uniformly in $B_R$ to $u_f(t,\cdot)$ and $\nabla_xu_{f_n}(t,\cdot)$, respectively. The
arbitrariness of $R$ and $T$ yields the assertion.
\end{proof}

\subsection{Hypercontractivity}
Throughout this and the forthcoming subsections we set
\begin{equation*}
{\mathcal F}(\zeta)=|\sqrt{Q}\nabla_x\zeta|^2,\qquad\;\,{\mathcal G}(\zeta)=\sum_{i=1}^d|\sqrt{Q}\nabla_xD_i\zeta|^2
\end{equation*}
for any smooth enough function $\zeta$. To begin with, we recall the following crucial result.

\begin{lemm}[Lemma 3.1 of \cite{AngLorLun}]
\label{derivative}
Assume that Hypotheses $\ref{base}$ hold true and fix $[a,b]\subset I$. If $f\in C^{1,2}_b([a,b]\times\Rd)$ and  $f(r,\cdot)$ is constant outside a compact set $K$ for every $r\in [a,b]$, then the function $r \mapsto \int_{\Rd}f(r,\cdot)d\mu_r$ is continuously differentiable in $[a,b]$ and
\begin{eqnarray*}
D_r\int_{\Rd}f(r,\cdot)d\mu_r=\int_{\Rd}D_rf(r,\cdot)d\mu_r -\int_{\Rd}\mathcal{A}(r)f(r, \cdot)d\mu_r,\qquad\;\,r\in [a,b].
\end{eqnarray*}
\end{lemm}

Further, we introduce the following additional assumptions.
\begin{hyp}
\label{cond-iper}
\begin{enumerate}[\rm(i)]
\item
$\psi\in B(I\times\Rd;{\rm Lip}(\R^{d+1}))\cap C(I\times\Rd\times\R\times\Rd)$, condition \eqref{hyp-holder} is satisfied in $[s,T]$, for any $T>s\in I$ and some
constant which may depend also on $s$ and $T$,
and there exist two constants $\xi_0\ge 0$ and $\xi_1$ such that
$u\psi(t,x,u,v)\le \xi_0|u|+\xi_1u^2+\xi_2|u||v|$ for any $t\ge s$, $x,v\in\Rd$ and $u\in\R$;
\item
there exists a non-negative function $\widetilde\varphi:\Rd\to\R$, blowing up at infinity such that
$\mathcal A\widetilde\varphi+k_1|\nabla\widetilde\varphi|\le a\widetilde\varphi$ in $\Rd$ for some locally bounded functions $a,k_1$;
\item
there exist locally bounded functions $C_0, C_1, C_2:I\to\R^+$ such that
\begin{align*}
&|Q(t,x)x|\leq C_0(t)|x|^3\varphi(x),\qquad\;\,{\rm Tr}(Q(t,x))\leq C_1(t)(1+|x|^2)\varphi(x),\notag\\
&\langle b(t,x),x\rangle\leq C_2(t)|x|^2\varphi(x)
\end{align*}
for any $t\in I$ and any $x\in\Rd$, where $\varphi$ is the Lyapunov function introduced in Hypothesis
$\ref{base}(iii)$;
\item
there exists a positive constant $K$ such that
\begin{align}
\int_{\Rd}|f|^q\log(|f|)d\mu_t\leq &\|f\|_{L^q(\Rd,\mu_t)}^q\log(\|f\|_{L^q(\Rd,\mu_t)})\notag\\
&+Kq\int_{\Rd}|f|^{q-2}|\nabla f|^2\one_{\{f\neq0\}}d\mu_t,
\label{LSI}
\end{align}
for any $t>s$, $f\in C^1_b(\Rd)$ and $q\in(1,+\infty)$.
\end{enumerate}
\end{hyp}

\begin{rmk}
\label{marsella}
{\rm
\begin{enumerate}[\rm (i)]
\item
Hypothesis \ref{cond-iper}(i) implies that $\psi(\cdot,\cdot,0,0)$ is bounded
in $[s,+\infty)\times\Rd$ and $\|\psi(\cdot,\cdot,0,0)\|_{C_b([s,+\infty)\times\Rd)}\le
\xi_0$.
\item
Sufficient conditions for \eqref{LSI} to hold are given in \cite{AngLorLun}. In particular, \eqref{LSI} holds true when \eqref{point_p_D11} is satisfied with $p=1$ (see Remark \ref{rmk-2.4}).
\end{enumerate}
}
\end{rmk}

We can now prove the main result of this subsection.

\begin{thm}
\label{vacanza}
Let Hypotheses $\ref{base}$ and $\ref{cond-iper}$ be satisfied. Then,
for any $f\in L^p(\Rd,\mu_s)$ $(p\ge p_0)$ and $t>s$ the function ${\mathcal N}(t,s)f$ belongs to $W^{1,p_{\gamma}(t)}(\Rd,\mu_t)$ and satisfies the estimates
\begin{align}
&\|{\mathcal N}(t,s)f\|_{L^{p_{\gamma}(t)}(\Rd,\mu_t)}\leq e^{\omega_{p,\gamma}(t-s)}
[\|f\|_{L^p(\Rd,\mu_s)}+\xi_0(t-s)],
\label{botti}
\\
&\|\nabla_x{\mathcal N}(t,s)f\|_{L^{p_{\gamma}(t)}(\Rd,\mu_t)}\leq c_0(t\!-\!s)
e^{\omega_{p,\sqrt{\gamma}}(t\!-\!s)}[\|f\|_{L^p(\Rd,\mu_s)}\!+\!\xi_0(t-s)]\!+\!
c_1(t-s)\xi_0,
\label{botti-1}
\end{align}
where $p_{\gamma}(t):=\gamma^{-1}(p-1)(e^{\kappa_0 K^{-1}(t-s)}-1)+p$ for any $\gamma>1$,
$\kappa_0$ being the ellipticity constant in Hypothesis $\ref{base}(ii)$ and $K$ being the constant in \eqref{LSI}, $\omega_{p,\sigma}=\xi_1+(\xi_2^+)^2\sigma[(\sigma-1)(p-1)\kappa_0]^{-1}$,
$\gamma'$ is given by \eqref{edoardo} and
the functions $c_0, c_1:(0,+\infty)\to\R^+$ are continuous and blow up at zero.
\end{thm}

\begin{proof}
To begin with, we observe that it suffices to prove \eqref{botti} and \eqref{botti-1}
for functions $f\in C_b^1(\Rd)$. Indeed, in the general case, the assertion follows
approximating $f$ with a sequence $(f_n)\subset C^1_b(\Rd)$ which converges to $f$ in
$L^p(\Rd,\mu_s)$. By \eqref{dip-p-weak-0}, ${\mathcal N}(t,s)f_n$ converges
to ${\mathcal N}(t,s)f$ in $W^{1,p}(\Rd,\mu_t)$ for almost every $t>s$.
Hence, writing \eqref{botti} and \eqref{botti-1} with $f$ being replaced by $f_n$
and letting $n$ tend to $+\infty$, the assertion follows at once by applying Fatou lemma.

We split the rest of the proof into two steps.
In the first one we prove \eqref{botti} and in the latter one \eqref{botti-1}.

{\em Step 1.} Fix $f\in C^1_b(\Rd)$, $n\in\N$, $\varepsilon>0$ and set
$\beta_{n,\varepsilon}(t):=\|v_{n,\varepsilon}(t,\cdot)\|_{p_{\gamma}(t)}$ for any $t>s$,
where $v_{n, \varepsilon}= (\vartheta^2_n {\mathcal N}(\cdot,s)f+\varepsilon)^{1/2}$ and
$\vartheta_n= \zeta(n^{-1}|x|)$ for any $x \in \Rd$ and $n \in \N$.
Here, $\zeta$ is a smooth function such that $\one_{B_1}\le\zeta \le\one_{B_2}$.
Moreover, we set $\varphi_{1,n}=|\zeta'(|\cdot|/n)|\varphi$, $\varphi_{2,n}=|\zeta''(|\cdot|/n)|\varphi$ for any $n\in\N$.
We recall that in \cite[Theorem 5.4]{KunLorLun09Non} it has been proved that $\sup_{t\in I}\|\varphi\|_{L^1(\Rd,\mu_t)}<+\infty$. Hence,
the functions $t\mapsto \|\varphi_{j,n}\|_{L^p(\Rd,\mu_t)}$ $(j=1,2)$ are bounded in $I$ and pointwise converge to zero as $n\to +\infty$.

By definition, the function $u=\mathcal N(\cdot,s)f$
belongs to $C^{0,1}_b([s,\tau]\times\Rd)$ for any $\tau>s$ and is a classical solution to problem \eqref{n-sm-pb}.
Moreover Lemma \ref{derivative} shows that $\beta_{n,\varepsilon}$ is differentiable in $(s,+\infty)$ and a straightforward computation reveals that
\begin{align*}
\beta_{n,\varepsilon}'(t)=&-\frac{p_{\gamma}'(t)}{p_{\gamma}(t)}\beta_{n,\varepsilon}(t)
\log(\beta_{n,\varepsilon}(t))\\
&+\frac{1}{p_{\gamma}(t)}(\beta_{n,\varepsilon}(t))^{1-p_{\gamma}(t)}
\int_{\Rd}\{D_t[(v_{n,\varepsilon}(t,\cdot))^{p_{\gamma}(t)}]
-\A(t)[(v_{n,\varepsilon}(t,\cdot)^{p_{\gamma}(t)}]\}d\mu_t.
\end{align*}
Taking into account that
\begin{align*}
&D_t[(v_{n,\varepsilon}(t,\cdot))^{p_{\gamma}(t)}]-\A[(v_{n,\varepsilon}(t,\cdot)^{p_{\gamma}(t)}]\\
=&p_{\gamma}'(t)(v_{n,\varepsilon}(t,\cdot))^{p_{\gamma}(t)}\log(v_{n,\varepsilon}(t,\cdot))
\!-\!p_{\gamma}(t)(p_{\gamma}(t)\!-\!1)(v_{n,\varepsilon}(t,\cdot))^{p_{\gamma}(t)-2}({\mathcal F}(v_{n,\varepsilon}))(t,\cdot)\\
&+p_{\gamma}(t)(v_{n,\varepsilon}(t,\cdot))^{p_{\gamma}(t)-1}(D_tv_{n,\varepsilon}(t,\cdot)
-\A(t) v_{n,\varepsilon}(t,\cdot))
\end{align*}
and
\begin{align*}
D_tv_{n,\varepsilon}
-\A v_{n,\varepsilon}=&\vartheta_n^2v_{n,\varepsilon}^{-1}
u\psi_u-\varepsilon\vartheta_n^2{\mathcal F}(u)
v_{n,\varepsilon}^{-3}-{\rm Tr}(QD^2\vartheta_n)v_{n,\varepsilon}^{-1}
\vartheta_nu^2-\varepsilon u^2v_{n,\varepsilon}^{-3}{\mathcal F}(\vartheta_n)\notag\\
&-\langle b,\nabla\vartheta_n\rangle v_{n,\varepsilon}^{-1}\vartheta_n u^2
-2(2\varepsilon u+\vartheta_n^2u^3)\langle Q\nabla\vartheta_n,\nabla_xu\rangle\vartheta_nv_{n,\varepsilon}^{-3}\\
=:& \vartheta_n^2v_{n,\varepsilon}^{-1}
u\psi_u-\varepsilon\vartheta_n^2{\mathcal F}(u)
v_{n,\varepsilon}^{-3}+g_{n,\varepsilon}(t,\cdot),
\end{align*}
we deduce
\begin{align*}
\beta_{n,\varepsilon}'(t)=&
(\beta_{n,\varepsilon}(t))^{1-p_{\gamma}(t)}
\int_{\Rd}(v_{n,\varepsilon}(t,\cdot))^{p_{\gamma}(t)-1}g_{n,\varepsilon}(t,\cdot)d\mu_t
-\frac{p_{\gamma}'(t)}{p_{\gamma}(t)}\beta_{n,\varepsilon}(t)
\log(\beta_{n,\varepsilon}(t))\notag\\
&-(p_{\gamma}(t)-1)(\beta_{n,\varepsilon}(t))^{1-p_{\gamma}(t)}\int_{\Rd}(v_{n,\varepsilon}(t,\cdot))^{p_{\gamma}(t)-2}
({\mathcal F}(v_{n,\varepsilon}))(t,\cdot)d\mu_t\notag\\
&+\frac{p_{\gamma}'(t)}{p_{\gamma}(t)}(\beta_{n,\varepsilon}(t))^{1-p_{\gamma}(t)}
\int_{\Rd}(v_{n,\varepsilon}(t,\cdot))^{p_{\gamma}(t)}\log(v_{n,\varepsilon}(t,\cdot))d\mu_t\notag\\
&+(\beta_{n,\varepsilon}(t))^{1-p_{\gamma}(t)}\int_{\Rd}(v_{n,\varepsilon}(t,\cdot))^{p_{\gamma}(t)-2}
\vartheta_n^2u(t,\cdot)\psi_u(t,\cdot)d\mu_t\notag\\
&-\varepsilon(\beta_{n,\varepsilon}(t))^{1-p_{\gamma}(t)}
\int_{\Rd}\vartheta_n^2(v_{n,\varepsilon}(t,\cdot))^{p_{\gamma}(t)-4}({\mathcal F}(u))(t,\cdot))d\mu_t.
\end{align*}

Using Hypotheses \ref{cond-iper}(i), (iv), the expression of the function $t\mapsto p_{\gamma}(t)$ and Hypothesis \ref{base}(ii), we can estimate
\begin{align}
\beta_{n,\varepsilon}'(t)\le &
(\beta_{n,\varepsilon}(t))^{1-p_{\gamma}(t)}
\int_{\Rd}(v_{n,\varepsilon}(t,\cdot))^{p_{\gamma}(t)-1}g_{n,\varepsilon}(t,\cdot)d\mu_t\notag\\
&+\xi_0(\beta_{n,\varepsilon}(t))^{1-p_{\gamma}(t)}\int_{\Rd}\vartheta_n^2|u(t,\cdot)|
(v_{n,\varepsilon}(t,\cdot))^{p_{\gamma}(t)-2}d\mu_t
+\xi_1\beta_{n,\varepsilon}(t)\notag\\
&-\varepsilon\xi_1(\beta_{n,\varepsilon}(t))^{1-p_{\gamma}(t)}
\int_{\Rd}(v_{n,\varepsilon}(t,\cdot))^{p_{\gamma}(t)-2}d\mu_t\notag\\
&-(p-1)(1-\gamma^{-1})(\beta_{n,\varepsilon}(t))^{1-p_{\gamma}(t)}
\int_{\Rd}(v_{n,\varepsilon}(t,\cdot))^{p_{\gamma}(t)-2}({\mathcal F}(v_{n,\varepsilon}))(t,\cdot)d\mu_t\notag\\
&+\xi_2^+(\beta_{n,\varepsilon}(t))^{1-p_{\gamma}(t)}\int_{\Rd}\vartheta_n^2(v_{n,\varepsilon}(t,\cdot))^{p_{\gamma}(t)-2}
|u(t,\cdot)||\nabla_xu(t,\cdot)|d\mu_t\notag\\
&-\varepsilon(\beta_{n,\varepsilon}(t))^{1-p_{\gamma}(t)}
\int_{\Rd}\vartheta_n^2(v_{n,\varepsilon}(t,\cdot))^{p_{\gamma}(t)-4}{\mathcal F}(u(t,\cdot))d\mu_t.
\label{bedbreakfast}
\end{align}
Further, since
${\mathcal F}(v_{n,\varepsilon})=\vartheta_n^2{\mathcal F}(\vartheta_n)u^4v_{n,\varepsilon}^{-2}
+\vartheta_n^4u^2{\mathcal F}(u)v_{n,\varepsilon}^{-2}+2\vartheta_n^2 u^3v_{n,\varepsilon}^{-2}\langle Q \nabla_xu,\nabla\vartheta_n\rangle$
and
\begin{align*}
&\int_{\Rd}(v(t, \cdot))^{p_{\gamma}(t)-4}\vartheta_n^2(u(t,\cdot))^3\langle Q(t,\cdot) \nabla_xu(t,\cdot),\nabla\vartheta_n\rangle d\mu_t\\
\le & \delta \int_{\Rd} (v(t, \cdot))^{p_{\gamma}(t)-4}\vartheta_n^4(u(t,\cdot))^2({\mathcal F}(u))(t,\cdot) d\mu_t\\
&+ \frac{1}{\delta} \int_{\Rd} (v(t, \cdot))^{p_{\gamma}(t)-4}(u(t,\cdot))^4({\mathcal F}(\vartheta_n))(t,\cdot)d\mu_t,
\end{align*}
it follows that
\begin{align*}
&\int_{\Rd}(v_{n,\varepsilon}(t,\cdot))^{p_{\gamma}(t)-2}({\mathcal F}(v_{n,\varepsilon}))(t,\cdot)d\mu_t\\
\ge &(1\!-\!\delta)\!\int_{\Rd}\vartheta_n^4(u(t,\cdot))^2
(v_{n,\varepsilon}(t,\cdot))^{p_{\gamma}(t)-4}({\mathcal F}(u))(t,\cdot)d\mu_t-C_{\varepsilon,\delta}(t)\!\int_{\Rd}\varphi_{1,n} d\mu_t
\end{align*}
for any $\delta>0$ and some continuous function $C_{\varepsilon, \delta}:[s,+\infty)\to \R^+$.
Moreover, applying H\"older and Young inequalities and Hypothesis \ref{base}(ii) we can infer that
\begin{align*}
&\int_{\Rd}\vartheta_n^2(v_{n,\varepsilon}(t,\cdot))^{p_{\gamma}(t)-2}
|u(t,\cdot)||\nabla_xu(t,\cdot)|d\mu_t\nonumber\\
\le &\frac{\delta_1}{\kappa_0}\int_{\Rd}\vartheta_n^4(v_{n,\varepsilon}(t,\cdot))^{p_{\gamma}(t)-4}
|u(t,\cdot)|^2({\mathcal F}(u))(t,\cdot)d\mu_t
+\frac{1}{4\delta_1}(\beta_{n,\varepsilon}(t,\cdot))^{p_{\gamma}(t)}
\end{align*}
for any $\delta_1>0$ and
\begin{align*}
\int_{\Rd}\vartheta_n^2u(t,\cdot)
(v_{n,\varepsilon}(t,\cdot))^{p_{\gamma}(t)-2}d\mu_t
\le \int_{\Rd}(v_{n,\varepsilon}(t,\cdot))^{p_{\gamma}(t)-1}d\mu_t
\le (\beta_{n,\varepsilon}(t,\cdot))^{p_{\gamma}(t)-1}.
\end{align*}
Hence,
\begin{align*}
\beta_{n,\varepsilon}'(t)\le &\xi_0+\bigg (\xi_1+\frac{\xi_2^+}{4\delta_1}\bigg )\beta_{n,\varepsilon}(t)+(\beta_{n,\varepsilon}(t))^{1-p_{\gamma}(t)}
\int_{\Rd}(v_{n,\varepsilon}(t,\cdot))^{p_{\gamma}(t)-1}g_{n,\varepsilon}(t,\cdot)d\mu_t\notag\\
&-[(p-1)(1-\gamma^{-1})(1-\delta)-\kappa_0^{-1}\xi_2^+\delta_1]
(\beta_{n,\varepsilon}(t))^{1-p_{\gamma}(t)}\notag\\
&\qquad\quad\times\int_{\Rd}\vartheta_n^4|u(t,\cdot)|^2(v_{n,\varepsilon}(t,\cdot))^{p_{\gamma}(t)-4}
({\mathcal F}(u))(t,\cdot)d\mu_t\notag\\
&+\widetilde C_{\varepsilon,\delta,p,\gamma}(t)(\beta_{n,\varepsilon}(t))^{1-p_{\gamma}(t)}
\int_{\Rd}\varphi_{1,n} d\mu_t\notag\\
&-\varepsilon(\beta_{n,\varepsilon}(t))^{1-p_{\gamma}(t)}
\int_{\Rd}\vartheta_n^2(v_{n,\varepsilon}(t,\cdot))^{p_{\gamma}(t)-4}({\mathcal F}(u))(t,\cdot)d\mu_t\notag\\
&-\varepsilon\xi_1(\beta_{n,\varepsilon}(t))^{1-p_{\gamma}(t)}
\int_{\Rd}(v_{n,\varepsilon}(t,\cdot))^{p_{\gamma}(t)-2}d\mu_t
\end{align*}
for some continuous function $\widetilde C_{\varepsilon,\delta,p,\gamma}:[s,+\infty)\to\R^+$.
Now, we estimate the integral term containing $g_n$. We begin by observing that
\begin{align*}
&-2\int_{\Rd}(2\varepsilon u(t,\cdot)+\vartheta_n^2(u(t,\cdot))^3)\langle Q(t,\cdot)\nabla\vartheta_n,\nabla_xu\rangle\vartheta_n(v_{n,\varepsilon}(t,\cdot))^{p(t)-4}d\mu_t\\
\le & 4\varepsilon\delta_2\int_{\Rd}\vartheta_n^2(v_{n,\varepsilon}(t,\cdot))^{p_{\gamma}(t)-4}({\mathcal F}(u))(t,\cdot)d\mu_t\\
&+\varepsilon\delta_2^{-1}\int_{\Rd}|u(t,\cdot)|^2(v_{n,\varepsilon}(t,\cdot))^{p_{\gamma}(t)-4}
({\mathcal F}(\vartheta_n))(t,\cdot)d\mu_t\\
&+\delta_2\int_{\Rd}\vartheta_n^4|u(t,\cdot)|^2(v_{n,\varepsilon}(t,\cdot))^{p_{\gamma}(t)-4}({\mathcal F}(u))(t,\cdot)d\mu_t\\
&+\delta_2^{-1}\int_{\Rd}\vartheta_n^2|u(t,\cdot)|^4(v_{n,\varepsilon}(t,\cdot))^{p_{\gamma}(t)-4}
({\mathcal F}(\vartheta_n))(t,\cdot)d\mu_t\\
\le & 4\varepsilon\delta_2\int_{\Rd}\vartheta_n^2(v_{n,\varepsilon}(t,\cdot))^{p_{\gamma}(t)-4}
({\mathcal F}(u))(t,\cdot)d\mu_t+\widetilde C_{\varepsilon,\delta_2}(t)\int_{\Rd}\varphi_{1,n} d\mu_t
\\
&+\delta_2\int_{\Rd}\vartheta_n^4|u(t,\cdot)|^2(v_{n,\varepsilon}(t,\cdot))^{p_{\gamma}(t)-4}({\mathcal F}(u))(t,\cdot)d\mu_t
\end{align*}
for some continuous function $\widetilde C_{\varepsilon,\delta_2}:[s,+\infty)\to\R^+$.
Moreover,
\begin{align*}
&-\int_{\Rd}(v_{n,\varepsilon}(t,\cdot))^{p_{\gamma}(t)-4}u(t,\cdot)[\vartheta_n{\mathcal A}(t)\vartheta_n
-\varepsilon u(t,\cdot)({\mathcal F}(\vartheta_n))(t,\cdot)]d\mu_t\\
\le & \overline C_{\varepsilon}(t)\int_{\Rd}(\varphi_{1,n}+\varphi_{2,n}) d\mu_t
\end{align*}
for some positive and continuous function $\overline C_{\varepsilon}:[s,+\infty)\to\R^+$.
Hence, replacing these estimates in \eqref{bedbreakfast}, we get
\begin{align}
\beta_{n,\varepsilon}'(t)\le &\xi_0+\bigg (\xi_1+\frac{\xi_2^+}{4\delta_1}\bigg )\beta_{n,\varepsilon}(t)
+\widehat C_{\varepsilon,\delta,\delta_2,p}(t)(\beta_{n,\varepsilon}(t))^{1-p_{\gamma}(t)}
\int_{\Rd}(\varphi_{1,n}+\varphi_{2,n})d\mu_t\notag\\
&-[(p-1)(1-\gamma^{-1})(1-\delta)-\kappa_0^{-1}\xi_2^+\delta_1-\delta_2]
(\beta_{n,\varepsilon}(t))^{1-p_{\gamma}(t)}\notag\\
&\qquad\quad\times\int_{\Rd}\vartheta_n^4|u(t,\cdot)|^2(v_{n,\varepsilon}(t,\cdot))^{p_{\gamma}(t)-4}
({\mathcal F}(u))(t,\cdot)d\mu_t\notag\\
&-\varepsilon(1-4\delta_2)(\beta_{n,\varepsilon}(t))^{1-p_{\gamma}(t)}
\int_{\Rd}\vartheta_n^2(v_{n,\varepsilon}(t,\cdot))^{p_{\gamma}(t)-4}({\mathcal F}(u))(t,\cdot)d\mu_t\notag\\
&-\varepsilon\xi_1(\beta_{n,\varepsilon}(t))^{1-p_{\gamma}(t)}
\int_{\Rd}(v_{n,\varepsilon}(t,\cdot))^{p_{\gamma}(t)-2}d\mu_t,
\label{giorgio}
\end{align}
where, again, $\widehat C_{\varepsilon,\delta,\delta_2,p}:(s,+\infty)\to\R^+$ is a continuous function.
Choosing $\delta=1/2$, $\delta_1=(p-1)(1-\gamma^{-1})\kappa_0/(4\xi_2)$, if $\xi_2>0$,  $\delta_1=0$, otherwise, and then $\delta_2$
small enough we obtain
\begin{align}
\beta_{n,\varepsilon}'(t)\le &\xi_0+\omega_{p,\gamma}\beta_{n,\varepsilon}(t)+
\widehat C_{\varepsilon,1/2,\delta_2,p}(t)(\beta_{n,\varepsilon}(t))^{1-p_{\gamma}(t)}
\|\varphi_{1,n}+\varphi_{2,n}\|_{L^1(\Rd,\mu_t)}\notag\\
&-\varepsilon\xi_1(\beta_{n,\varepsilon}(t))^{1-p_{\gamma}(t)}
\int_{\Rd}(v_{n,\varepsilon}(t,\cdot))^{p_{\gamma}(t)-2}d\mu_t.
\label{giorgio-1}
\end{align}
Hence, integrating \eqref{giorgio-1} between $s$ and $t$ and letting first $n\to +\infty$ and then $\varepsilon\to 0^+$, by dominated convergence we get
\begin{align*}
\|u(t,\cdot)\|_{L^{p_{\gamma}(t)}(\Rd,\mu_t)}\le\|f\|_{L^p(\Rd,\mu_s)}+\xi_0(t-s)
+\omega_{p,\gamma}\int_s^t\|u(r,\cdot)\|_{L^{p(r)}(\Rd,\mu_r)}dr.
\end{align*}
Applying the Gronwall Lemma we conclude the proof of \eqref{botti}.

{\em Step 2.} To check estimate \eqref{botti-1},
we arbitrarily fix $\gamma\in (1,+\infty)$, $t>s$ and we take
\begin{equation}
\varepsilon=\frac{K}{2\kappa_0}\log\bigg (\frac{\gamma e^{\kappa_0K^{-1}(t-s)}}{\gamma+e^{\kappa_0K^{-1}(t-s)}-1}\bigg ),\qquad\;\,
\gamma'=\gamma\frac{e^{\kappa_0K^{-1}(t-s-\varepsilon)}-1}{e^{\kappa_0K^{-1}(t-s)}-1}.
\label{edoardo}
\end{equation}
With these choices of $\varepsilon$ and $\gamma'$, we have
$p_{\gamma'}(t-\varepsilon)=p_{\gamma}(t)$.
From Step 1, we know that
${\mathcal N}(t-\varepsilon,s)f\in L^{p_{\gamma'}(t-\varepsilon)}(\Rd,\mu_{t-\varepsilon})$ and
\begin{equation}
\|{\mathcal N}(t-\varepsilon,s)f\|_{L^{p_{\gamma}(t)}(\Rd,\mu_{t-\varepsilon})}\leq
e^{\omega_{p,\gamma'}(t-s-\varepsilon)}(\|f\|_{L^p(\Rd,\mu_s)}+\xi_0(t-s)).
\label{magliarosa}
\end{equation}
By the evolution law and estimates \eqref{magliarosa} and \eqref{ciclismo} we get
\begin{align*}
&\sup_{\tau \in (t-\varepsilon,T)}
\sqrt{\tau-t+\varepsilon}\,\|\nabla_x {\mathcal N}(\tau,s)f\|_{L^{p_{\gamma}(t)}(\Rd,\mu_{\tau})}\\
\le &C_{T-t+\varepsilon}\{
e^{\omega_{p,\gamma'}(t-s-\varepsilon)}(\|f\|_{L^p(\Rd,\mu_s)}+\xi_0(t-s))+
(T-t+\varepsilon+1)\|\psi(\cdot,\cdot,0,0)\|_{\infty}\}
\end{align*}
for any $T>t-\varepsilon$.
In particular, taking $T=t$ and using Remark \ref{marsella}(i) to estimate $\|\psi(\cdot,\cdot,0,0)\|_{\infty}\le\xi_0$, we get
\begin{align}
\|\nabla_x {\mathcal N}(t,s)f\|_{L^{p_{\gamma}(t)}(\Rd,\mu_t)}\le &
\frac{C_{\varepsilon}}{\sqrt{\varepsilon}}e^{-\varepsilon\omega_{p,\gamma'}}
e^{\omega_{p,\gamma'}(t-s)}[\|f\|_{L^p(\Rd,\mu_s)}+\xi_0(t-s)]\notag\\
&+\bigg (\sqrt{\varepsilon}+\frac{1}{\sqrt{\varepsilon}}\bigg ) C_{\varepsilon}\xi_0.
\label{raffaele}
\end{align}
Replacing the value of $\varepsilon$ in the expression of $\gamma'$ (see \eqref{edoardo}),
we deduce that
\begin{eqnarray*}
\gamma'
\ge\inf_{\delta\ge 1}
\gamma(\delta-1)^{-1}\bigg [\bigg (\frac{\delta(\gamma+\delta-1)}{\gamma}\bigg )^{1/2}-1\bigg ]=\sqrt{\gamma}
\end{eqnarray*}
and, since the function $\sigma\mapsto \omega_{p,\sigma}$ is decreasing,
$\omega_{p,\gamma'}\le\omega_{p,\sqrt{\gamma}}$. Finally, observing that $e^{-\varepsilon\omega_{p,\gamma'}}$
is bounded in $(s,+\infty)$, $\varepsilon<(2\kappa_0)^{-1}K\log(\gamma)$ (which follows from \eqref{edoardo} recalling that $\gamma'\ge\sqrt{\gamma}$)
and $\varepsilon\sim (2\gamma)^{-1}(\gamma-1)(t-s)$ as $t-s\to 0^+$,
formula \eqref{botti-1} follows immediately replacing in
\eqref{raffaele} the value of $\varepsilon$ given by \eqref{edoardo}.
\end{proof}

\begin{rmk}
{\rm As the proof of Theorem \ref{vacanza} shows, if $\xi_2\le 0$, then we can take $\gamma=1$ and
$\omega_{p,1}=\xi_1$ in \eqref{botti}.}
\end{rmk}

\subsection{Supercontractivity}

In the next theorem we prove a stronger result than Theorem \ref{vacanza}, i.e., we
prove that the nonlinear evolution operator ${\mathcal N}(t,s)$ satisfies a supercontractivity property.
For this purpose, we introduce the following additional assumption.
\begin{hypo}
\label{hyp-super}
There exists a decreasing function $\nu:(0,+\infty)\to\R^+$ blowing up as $\sigma$ tends to $0^+$ such that
\begin{align}
&\int_{\Rd}|f|^p\log(|f|)d\mu_r
-\|f\|_{L^p(\Rd,\mu_r)}^p\log(\|f\|_{L^p(\Rd,\mu_r)})\notag\\
\le & \frac{\nu(\sigma)}{p}\|f\|_{L^p(\Rd,\mu_r)}^p
+\sigma p\int_{\Rd}|f|^{p-2}|\nabla f|^2\one_{\{f\neq 0\}}d\mu_r
\label{LSI_epsilon}
\end{align}
for any $r\in I$, $\sigma>0$ and $f\in C^1_b(\Rd)$.
\end{hypo}

\begin{rmk}
{\rm Sufficient conditions for \eqref{LSI_epsilon} to hold are given in \cite{AngLorOnI}. In particular, it holds true when \eqref{point_p_D11} is satisfied with $p=1$ (see Remark \ref{rmk-2.4}) and there exist $K>0$ and $R>1$ such that $\langle b(t,x),x\rangle \le -K|x|^2\log|x|$ for any $t \in I$ and $|x|\ge R$.}
\end{rmk}

\begin{thm}
Let Hypotheses $\ref{base}$, $\ref{cond-iper}(i)$-$(iii)$ and $\ref{hyp-super}$ be satisfied.
Then, for any $t>s \in I$, $p_0\le p<q<+\infty$ and any $f \in L^p(\Rd, \mu_s)$, ${\mathcal N}(t,s)f$
belongs to $W^{1,q}(\Rd, \mu_t)$ and
\begin{align}
&\|{\mathcal N}(t,s)f\|_{L^q(\Rd, \mu_t)} \le c_2(t-s)(\|f\|_{L^p(\Rd,\mu_s)}+\xi_0(t-s)),
\label{gotico}\\[1mm]
&\|\nabla_x{\mathcal N}(t,s)f\|_{L^q(\Rd, \mu_t)} \le c_3(t-s)\|f\|_{L^p(\Rd,\mu_s)}+c_4(t-s)\xi_0.
\label{gotico-1}
\end{align}
Here, $c_2, c_3,c_4:(0,+\infty)\to \R^+$ are continuous functions such that $\lim_{r \to 0^+}c_k(r)=+\infty$ $(k=2,3,4)$.
\end{thm}
\begin{proof}
The proof of this result follows the same lines of the proof of Theorem \ref{vacanza}. For this reason we use the notation therein introduced and we limit ourselves to sketching it in the case when $f\in C^1_b(\Rd)$.

{\em Step 1.} Here, we prove \eqref{gotico}. For any $\sigma>0$ and any $t\ge s$ we set $p(t)=e^{\kappa_0(2\sigma)^{-1}(t-s)}(p-1)+1$,
$m(t)= \nu(\sigma)(p^{-1}-(p(t))^{-1})$ and $\zeta_{n, \varepsilon}(t)= e^{-m(t)}\beta_{n, \varepsilon}(t)$. The function $\zeta_{n, \varepsilon}$ is differentiable in $(s,+\infty)$ and arguing as in the proof of the quoted theorem, using \eqref{LSI_epsilon} instead of \eqref{LSI} and the definition of $m(t)$ and $p(t)$, we deduce that
\begin{align*}
\zeta_{n, \varepsilon}'(t)=
\bigg [&(\beta_{n,\varepsilon}(t))^{1-p(t)}
\int_{\Rd}(v_{n,\varepsilon}(t,\cdot))^{p(t)-1}g_{n,\varepsilon}(t,\cdot)d\mu_t\notag\notag\\
&-\frac{p-1}{2}(\beta_{n,\varepsilon}(t))^{1-p(t)}
\int_{\Rd}(v_{n,\varepsilon}(t,\cdot))^{p(t)-2}
({\mathcal F}(v_{n,\varepsilon}))(t,\cdot)d\mu_t\notag\\
&+(\beta_{n,\varepsilon}(t))^{1-p(t)}\int_{\Rd}(v_{n,\varepsilon}(t,\cdot))^{p(t)-2}
\vartheta_n^2u(t,\cdot)\psi_u(t,\cdot)d\mu_t\notag\\
&-\varepsilon(\beta_{n,\varepsilon}(t))^{1-p(t)}
\int_{\Rd}\vartheta_n^2(v_{n,\varepsilon}(t,\cdot))^{p(t)-4}({\mathcal F}(u))(t,\cdot)d\mu_t\bigg ]e^{-m(t)}
\end{align*}
and the same arguments used to prove \eqref{giorgio} show that, if $\delta_2<1/4$, then
\begin{align*}
\zeta_{n,\varepsilon}'(t)\le &
\xi_0e^{-m(t)}+\bigg (\xi_1+\frac{\xi_2^+}{4\delta_1}\bigg )
\zeta_{n,\varepsilon}(t)\\
&+\bigg (
\widehat C_{\varepsilon,\delta,\delta_2,p}(t)(\beta_{n,\varepsilon}(t))^{1-p(t)}
\int_{\Rd}(\varphi_{1,n}+\varphi_{2,n})d\mu_t\notag\\
&\qquad\;-[2^{-1}(p-1)(1-\delta)-\kappa_0^{-1}\xi_2^+\delta_1-\delta_2]
(\beta_{n,\varepsilon}(t))^{1-p(t)}\notag\\
&\qquad\qquad\;\times\int_{\Rd}\vartheta_n^4|u(t,\cdot)|^2(v_{n,\varepsilon}(t,\cdot))^{p(t)-4}
({\mathcal F}(u))(t,\cdot)d\mu_t\notag\\
&\qquad\;-\varepsilon\xi_1(\beta_{n,\varepsilon}(t))^{1-p(t)}
\int_{\Rd}(v_{n,\varepsilon}(t,\cdot))^{p(t)-2}d\mu_t\bigg )e^{-m(t)}.
\end{align*}
Choosing $\delta=1/2$, $\delta_1=(p-1)\kappa_0(8\xi_2)^{-1}$, if $\xi_2>0$, $\delta_1=0$, otherwise, and $\delta_2=[(p-1)\wedge 2]/8$ we get
\begin{align}
\zeta_{n,\varepsilon}'(t)\le \widetilde\omega_p\zeta_{n,\varepsilon}(t)
+e^{-m(t)}\bigg [&\xi_0+
\widehat C_{\varepsilon,\delta_2,p}(t)(\beta_{n,\varepsilon}(t))^{1-p(t)}
\|\varphi_{1,n}+\varphi_{2,n}\|_{L^1(\Rd,\mu_t)}\notag\\
&-\varepsilon\xi_1(\beta_{n,\varepsilon}(t))^{1-p(t)}
\int_{\Rd}(v_{n,\varepsilon}(t,\cdot))^{p(t)-2}d\mu_t\bigg ],
\label{nanna}
\end{align}
where $\widetilde\omega_p=\xi_1+2(\xi_2^+)^2(\kappa_0(p-1))^{-1}$.
Hence, integrating \eqref{nanna} between $s$ and $t$ and letting first $n\to +\infty$ and then $\varepsilon\to 0^+$, by dominated convergence we get
\begin{align*}
e^{-m(t)}\|u_f(t,\cdot)\|_{L^{p(t)}(\Rd,\mu_t)}\le &\xi_0(t-s)+\|f\|_{L^p(\Rd,\mu_s)}\\
&+\widetilde\omega_p\int_s^te^{-m(r)}\|u(r,\cdot)\|_{L^{p(r)}(\Rd,\mu_r)}dr,
\end{align*}
which yields
$\|u_f(t,\cdot)\|_{L^{p(t)}(\Rd,\mu_t)}\le
e^{\widetilde\omega_p(t-s)+m(t)}(\xi_0(t-s)+\|f\|_{L^p(\Rd, \mu_s)})$.
Now, for any $q>p$ and $t>s$, we fix $\sigma= \kappa_0(t-s)(2\log(q-1)-2\log(p-1))^{-1}$. We get $p(t)=q$ and from
the previous inequality the claim follows with
\begin{eqnarray*}
c_2(r)=\exp(\widetilde\omega_pr+ (p^{-1}-q^{-1})\nu(\kappa_0r
(2\log(q-1)-2\log(p-1))^{-1})).
\end{eqnarray*}

{\em Step 2.} Fix $q>p$. By Step 1, ${\mathcal N}((t+s)/2,s)f$ belongs to $L^q(\Rd,\mu_{(t+s)/2})$ and
\begin{eqnarray*}
\|{\mathcal N}((t+s)/2,s)f\|_{L^q(\Rd,\mu_{(t+s)/s})}\le c_2((t-s)/2)\bigg (\|f\|_{L^p(\Rd,\mu_s)}
+\xi_0\frac{t-s}{2}\bigg ).
\end{eqnarray*}
The same arguments used in Step 2 of the proof of Theorem \ref{vacanza} show that
${\mathcal N}(t,s)f\in W^{1,q}(\Rd,\mu_{\tau})$ for any $\tau>(t+s)/2$ and
\begin{align*}
&\sqrt{\frac{t-s}{2}}
\|\nabla_x{\mathcal N}(t,s)f\|_{L^q(\Rd,\mu_{\tau})}\notag\\
\le & C_{(t-s)/2}\bigg [c_2((t-s)/2)\bigg (\|f\|_{L^p(\Rd,\mu_s)}+\xi_0\frac{t-s}{2}\bigg )
+\bigg (\frac{t-s}{2}+1\bigg )\xi_0\bigg ]
\end{align*}
Estimate \eqref{gotico-1} follows with $c_3(r)=\sqrt{2/r}C_{r/2}c_2(r/2)$, $c_4(r)=C_{r/2}[c_2(r/2)\sqrt{r/2}$ $+\sqrt{r/2}+\sqrt{2/r}]$.
\end{proof}

\subsection{Ultraboundedness}
To begin with, we prove a sort of Harnack inequality, which besides the interest in its own will be crucial to prove the ultraboundedness of the nonlinear evolution operator
${\mathcal N}(t,s)$.

\begin{prop}
\label{harnack}
Let Hypotheses $\ref{base}(i)$-$(iii)$, $\ref{cond-iper}(i)$-$(iii)$ be satisfied. Further, suppose that estimate $\eqref{point_p_D11}$ holds, with $p=1$ and some constant $\sigma_1\in\R$.
Then, for any $f\in C_b(\Rd)$, $p>1$, $t>s$ and $x,y\in\Rd$ the following estimate holds true:
\begin{align}
|({\mathcal N}(t,s)f)(x)|^p\le& \exp\bigg (p(1+\xi_1^+)(t-s)+p\Theta(t-s)
\frac{(|x-y|+\xi_2^+(t-s))^2}{4\kappa_0(t-s)^2(p-1)}\bigg )\notag\\
&\quad\;\times [(G(t,s)|f|^p)(y)+\xi_0^p],
\label{cavo}
\end{align}
where $\Theta(r)=(e^{2\sigma_1r}-1)/(2\sigma_1)$, if $\sigma_1>0$ and $\Theta(r)=r$ otherwise.
\end{prop}

\begin{proof}
To begin with, we observe that it suffices to prove \eqref{cavo} for
functions in $C^1_b(\Rd)$. Indeed, if $f\in C_b(\Rd)$, we can
determine a sequence $(f_n)\subset C^1_b(\Rd)$, bounded with respect to the sup-norm
and converging to $f$ locally uniformly in $\Rd$.
Writing \eqref{cavo} with $f$ replaced by $f_n$ and using Theorem \ref{compleanno}
and \cite[Proposition 3.1(i)]{KunLorLun09Non}, we can let $n$ tend to $+\infty$ and complete the proof.

So, let us fix $f\in C^1_b(\Rd)$ and set
$\Phi_n(r):=[G(t,r)(\vartheta_n^2v_\varepsilon(r,\cdot))](\phi(r))+\xi_0^p$ for any $n\in\N$ and $r\in (s,t)$,
where $v_\varepsilon=(u_f^2+\varepsilon)^{p/2}$, $u_f={\mathcal N}(\cdot,s)f$ (see Theorem \ref{prop-3.8}),
$\phi(r)=(r-s)(t-s)^{-1}x+(t-r)(t-s)^{-1}y$ and $(\vartheta_n)$ is a standard sequence of
cut-off functions. We note that $\Phi_n(r)\geq C_\Phi>0$ for any $r\in[s,t]$ and any $n\geq n_0$. This is clear if $\xi_0>0$.
Suppose that $\xi_0=0$.
If $r<t$ then $\Phi_n(r)$ is positive since $v_\varepsilon>0$. If $r=t$, then
$\Phi_n(t)=(\vartheta_n(x))^2v_\varepsilon(t,x)$ which is positive if we choose
$n\in\N$ large enough such that $x\in {\rm supp}(\vartheta_n)$.
Moreover, $\Phi_n\in C^1((s,t))$. Hence $\log(\Phi_n)\in C^1((s,t))$ and we have
\begin{align*}
\frac{d}{dr}\log(\Phi_n(r))
= \frac{1}{\Phi_n(r)}\{&[G(t,r)(\vartheta_n^2D_tv_\varepsilon(r,\cdot))-
\A(\vartheta_n^2v_\varepsilon(r,\cdot)))]
(\phi(r)) \\
& +(t-s)^{-1}\langle [\nabla_xG(t,r)(\vartheta_n^2v_\varepsilon(r,\cdot))](\phi(r)),x-y\rangle\}.
\end{align*}
We observe that
\begin{align*}
D_t(\vartheta_n^2v_\varepsilon)-\A(\vartheta_n^2v_{\varepsilon})
= & p\vartheta_n^2(u_f^2+\varepsilon)^{\frac{p}{2}-1}u_f\psi_{u_f}
-p\vartheta_n^2v_{\varepsilon}^{1-\frac{4}{p}}((p-1)u_f^2+\varepsilon){\mathcal F}(u_f)\\
& -4pv_{\varepsilon}^{1-\frac{2}{p}}\vartheta_nu_f\langle Q\nabla\vartheta_n,\nabla_xu_f\rangle
-2\vartheta_nv_\varepsilon \A\vartheta_n-2v_{\varepsilon}{\mathcal F}(\vartheta_n),
\end{align*}
and
\begin{align*}
|\nabla_xG(t,r)(\vartheta_n^2v_{\varepsilon}(r,\cdot))|
\leq & e^{\sigma_1(t-r)}G(t,r)|\nabla_x(\vartheta_n^2v_\varepsilon(r,\cdot))|\\
\leq & pe^{\sigma_1(t\!-\!r)}G(t,r)(\vartheta_n^2(v_{\varepsilon}(r,\cdot))^{1\!-\!\frac{2}{p}}|u_f(r,\cdot)|\kappa_0^{-\frac{1}{2}}\!(({\mathcal F}(u_f))(r,\cdot))^{\frac{1}{2}})\\
&+e^{\sigma_1(t-r)}G(t,r)(2\vartheta_n|\nabla\vartheta_n|v_\varepsilon(r,\cdot)).
\end{align*}
Hence, we get
\begin{align*}
&\frac{d}{dr}\log\Phi_n(r) \\
\leq & \frac{1}{\Phi_n(r)}
\bigg\{p\frac{|x-y|}{t-s}e^{\sigma_1(t-r)}G(t,r)[\vartheta_n^2
(v_{\varepsilon}(r,\cdot))^{1-\frac{2}{p}}|u_f(r,\cdot)|\kappa_0^{-1/2}(({\mathcal F}(u_f))(r,\cdot))^{\frac{1}{2}}]\\
&\qquad\quad\;\;\,-G(t,r)\zeta_{n,\varepsilon}(r,\cdot) +\frac{|x-y|}{t-s}e^{\sigma_1(t-r)}G(t,r)(2\vartheta_n|\nabla\vartheta_n|
v_\varepsilon(r,\cdot))\bigg\}(\phi(r)),
\end{align*}
where
\begin{align*}
\zeta_{n,\varepsilon}=&
2\vartheta_n(\A\vartheta_n)v_{\varepsilon}+2{\mathcal F}(\vartheta_n) v_\varepsilon+4p\vartheta_nv_{\varepsilon}^{1-\frac{2}{p}}u_f\langle Q\nabla\vartheta_n,\nabla_xu_f\rangle\\
&-p\vartheta_n^2v_{n,\varepsilon}^{1-\frac{2}{p}}u_f\psi_{u_f}
+p\vartheta_n^2 v_{\varepsilon}^{1-\frac{4}{p}}((p-1)u_f+\varepsilon)
{\mathcal F}(u_f).
\end{align*}
From Hypothesis \ref{cond-iper}(i) it follows that
\begin{align}
\frac{d}{dr}\log\Phi_n(r)
\leq & \frac{1}{\Phi_n(r)}
G(t,r)\bigg\{-2\vartheta_n(\A(r)\vartheta_n)v_{\varepsilon}(r,\cdot)
+p\xi_0\vartheta_n^2v_{\varepsilon}^{1-\frac{1}{p}}
+\xi_1^+\vartheta_n^2pv_{\varepsilon}(r,\cdot)\notag\\
&\qquad\qquad\qquad\;+4p(v_{\varepsilon}(r,\cdot))^{1-\frac{2}{p}}\vartheta_n|u_f(r,\cdot)|
|\langle Q(r,\cdot)\nabla\vartheta_n,\nabla_xu_f(r,\cdot)\rangle|\notag\\
&\qquad\qquad\qquad\; - p\vartheta_n^2v_\varepsilon(r,\cdot)
\Big[((p-1)(u_f(r,\cdot))^2+\varepsilon)(h_{\varepsilon}(r,\cdot))^2 \notag\\
&\qquad\qquad\qquad\; -|u_f(r,\cdot)|h_{\varepsilon}(r,\cdot)\frac{e^{\sigma_1(t-r)}|x-y|+\xi_2^+(t-s)}{\sqrt{\kappa_0} (t-s)}\Big]\bigg\}(\phi(r))\notag\\
&+\frac{|x-y|}{t-s}e^{\sigma_1(t-r)}\{G(t,r)[|\nabla\vartheta_n|v_\varepsilon(r,\cdot)]
\}(\phi(r)),
\label{999}
\end{align}
where $h_\varepsilon= (u_f^2+\varepsilon)^{-1}\sqrt{{\mathcal F}(u_f)}$.
Using the Cauchy-Schwarz inequality we can estimate
\begin{align*}
v_{\varepsilon}^{1-\frac{2}{p}}\vartheta_n|u_f|
|\langle Q\nabla\vartheta_n,\nabla_xu_f\rangle|
\le &v_{\varepsilon}^{1-\frac{2}{p}}\vartheta_n|u_f|
\sqrt{{\mathcal F}(\vartheta_n)}
\sqrt{{\mathcal F}(u_f)}\notag\\
\le &\delta \vartheta_n^2v_{\varepsilon}h_{\varepsilon}^2u_f^2+
\frac{1}{4\delta}v_{\varepsilon}{\mathcal F}(\vartheta_n).
\end{align*}
Moreover, using formula \eqref{repres-formula}, we can estimate
\begin{align*}
(G(t,r)(\vartheta_n^2v_{\varepsilon}^{1-1/p}))(\phi(r))
\le &((G(t,r)(\vartheta_n^2v_{\varepsilon}))(\phi(r)))^{1-\frac{1}{p}}
((G(t,r)\vartheta_n)(\phi(r))^p\\
\le &((G(t,r)(\vartheta_n^2v_{\varepsilon}))(\phi(r)))^{1-\frac{1}{p}}\le
(\Phi_n(r))^{1-\frac{1}{p}}.
\end{align*}
These two estimates replaced in \eqref{999} give
\begin{align}
\frac{d}{dr}\log\Phi_n(r)
\leq & \frac{1}{\Phi_n(r)}\bigg (
G(t,r)\bigg\{\![p\delta^{-1}\!({\mathcal F}(\vartheta_n))(r,\cdot)\!-\!2\vartheta_n\A(r)\vartheta_n]
v_{\varepsilon}(r,\cdot)\!+\!\xi_1^+\vartheta_n^2pv_{\varepsilon}(r,\cdot)\notag\\
& - p\vartheta_n^2v_\varepsilon(r,\cdot)
\Big[((p-1-\delta)(u_f(r,\cdot))^2+\varepsilon)(h_{\varepsilon}(r,\cdot))^2 \notag\\
&\qquad\quad\qquad\quad -|u_f(r,\cdot)|h_{\varepsilon}(r,\cdot)\frac{e^{\sigma_1(t-r)}|x-y|\!+\!\xi_2^+(t-s)}{\sqrt{\kappa_0} (t-s)}\Big]\bigg\}\bigg )(\phi(r))\notag\\
&+\frac{|x-y|}{t-s}e^{\sigma_1(t-r)}\{G(t,r)[|\nabla\vartheta_n|v_\varepsilon(r,\cdot)]
\}(\phi(r))+p.
\label{tabloni}
\end{align}

Straightforward computations show that
${\mathcal A}(r)\vartheta_n$ and $({\mathcal F}(\vartheta_n))(r,\cdot)$ vanish pointwise in $\Rd$ as $n\to +\infty$, for
any $r\in (s,t)$ and there exists a positive constant
$C$ such that $|{\mathcal A}(r)\vartheta_n|+({\mathcal F}(\vartheta_n))(r,\cdot)\le C\varphi$ in $\Rd$ for any $n\in\N$, thanks to Hypothesis \ref{cond-iper}(iii).
By \cite[Lemma 3.4]{KunLorLun09Non} the function $G(t,\cdot)\varphi$ is bounded in $(s,t)\times B_R$ for any $R>0$. Hence, by dominated convergence we conclude that
$G(t,r)(\vartheta_n(\A(r)\vartheta_n)v_{\varepsilon}(r,\cdot))$ vanishes as $n\to +\infty$, pointwise in $\Rd$, for any $r\in (s,t)$ and
\begin{align}
&\|G(t,r)[p\delta^{-1}({\mathcal F}(\vartheta_n))(r,\cdot)-2\vartheta_n\A(r)\vartheta_n]\|_{C_b(B_R)}\notag\\
\le &C_{\delta,p,\|u_f\|_{\infty}}\sup_{r\in (s,t)}\|G(t,r)\varphi\|_{C_b(B_R)},
\label{co'}
\end{align}
where $R>\max\{|x|,|y|\}$. Similarly, the last but one term in \eqref{tabloni} vanishes pointwise
in $\Rd$ as $n\to +\infty$, for any $r\in (s,t)$ and
\begin{align}
|(G(t,r)[|\nabla\vartheta_n|v_\varepsilon(r,\cdot)])(\phi(r))|\le C_{p,\|u_f\|_{\infty}}\sup_{r\in (s,t)}\|G(t,r)\varphi\|_{C_b(B_R)}.
\label{ciafaloni}
\end{align}

Moreover, using the inequality
$\alpha\beta^2-\gamma\beta\geq -\gamma^2/(4\alpha)$ for any $\alpha>0$ and $\beta, \gamma\in\R$,
and that $G(t,s)g_1\leq G(t,s)g_2$ for any $t>s$ and any $g_1\leq g_2$, we deduce
\begin{align*}
\frac{d}{dr}\log(\Phi_n(r))
\leq &\frac{1}{\Phi_n(r)}(G(t,r)[p\delta^{-1}({\mathcal F}(\vartheta_n))(r,\cdot)-2\vartheta_n\A(r)\vartheta_n]
v_{\varepsilon}(r,\cdot))(\phi(r))\notag\\
& +p(1+\xi_1^{+}+e^{2\sigma_1^+(t\!-\!r)}\chi_{\delta})\notag\\
&+\frac{e^{\sigma_1(t-r)}|x-y|}{(t-s)\Phi_n(r)}(G(t,r)[|\nabla\vartheta_n|v_\varepsilon(r,\cdot)])
(\phi(r)),
\end{align*}
where $\chi_{\delta}=(|x-y|+\xi_2^+(t-s))^2(4\kappa_0(t-s)^2(p-1-\delta))^{-1}$.
Integrating both sides of the previous inequality in $(s,t)$ and taking \eqref{co'} and \eqref{ciafaloni} into account to let $n\to +\infty$, we get
\begin{align*}
\log\bigg (\frac{((u_f(t,x))^2+\varepsilon)^{p/2}+\xi_0^p}
{(G(t,s)((f^2+\varepsilon)^{p/2}))(y)+\xi_0^p}\bigg )
\leq & p[(1+\xi_1^+)(t-s)+\Theta(t-s)\chi_{\delta}],
\end{align*}
or even
\begin{align*}
((u_f(t,x))^2+\varepsilon)^{p/2}\le &\exp[p(1+\xi_1^+)(t-s)+p\Theta(t-s)\chi_{\delta}]\notag\\
&\quad\;\,\times[(G(t,s)((f^2+\varepsilon)^{p/2}))(y)+\xi_0^p].
\end{align*}
By \eqref{repres-formula} we can let $\varepsilon$ and $\delta$ tend to zero in both sides
of the previous inequality and this yields the assertion.
\end{proof}

We can now prove the main result of this subsection. For this purpose, we
set $\varphi_{\lambda}(x)=e^{\lambda |x|^2}$ for any $x\in\Rd$ and $\lambda>0$, and
introduce the following additional assumption.

\begin{hypo}
\label{salvadori}
For any $I\ni s<t$ and $\lambda>0$, the function $G(t,s)\varphi_{\lambda}$ belongs to $L^{\infty}(\Rd)$ and, for any $\delta>0$,
$+\infty>M_{\delta,\lambda}:=\sup_{t-s\ge \delta}\|G(t,s)\varphi_{\lambda}\|_{\infty}$.
\end{hypo}

\begin{rmk}
{\rm A sufficient condition for Hypothesis \ref{salvadori} to hold is given in \cite[Theorem 4.3]{AngLorOnI}. More precisely, it holds when $\eqref{point_p_D11}$ holds with $p=1$ and there exists
$K>0$, $\beta,R>1$ such that $\langle b(t,x),x\rangle\le -K|x|^2(\log(|x|))^{\beta}$ for any $t\in I$ and $x\in\Rd\setminus B_R$.}
\end{rmk}
\begin{thm}
Assume that Hypothesis $\ref{salvadori}$ and the conditions in Proposition $\ref{harnack}$ are
satisfied. Then, for any $I\ni s<t$, $f\in L^p(\Rd,\mu_s)$ $(p\in [p_0,+\infty))$ the function
${\mathcal N}(t,s)f$ belongs to $W^{1,\infty}(\Rd)$ and
\begin{align}
&\|{\mathcal N}(t,s)f\|_{\infty}\le c_4(t-s)\|f\|_{L^p(\Rd,\mu_s)}+c_5(t-s)\xi_0,
\label{giovanni}\\[1mm]
&\|\nabla_x{\mathcal N}(t,s)f\|_{\infty}\le c_6(t-s)\|f\|_{L^p(\Rd,\mu_s)}+c_7(t-s)\xi_0
\label{mancarella}
\end{align}
for some continuous functions $c_k:(0,+\infty)\to\R^+$ $(k=4,5,6,7)$ which blow up at zero.
\end{thm}

\begin{proof}
As usually, we prove the assertion for functions in $C^1_b(\Rd)$.

{\em Step 1.} Here, we prove \eqref{giovanni}.
So, let us fix $f\in C^1_b(\Rd)$ and $x\in\Rd$.
By the invariance property of the family $\{\mu_t: t\in I\}$ and inequality \eqref{cavo},
we can estimate
\begin{align*}
&\|f\|_{L^p(\Rd,\mu_s)}^p=\int_{\Rd}(G(t,s)|f|^p)(y)\mu_t(dy)\\
\ge &\int_{B_R}[(G(t,s)|f|^p)(y)+\xi_0^p]\mu_t(dy)-\xi_0^p\\
\ge & |({\mathcal N}(t,s)f)(x)|^pe^{-p\phi(t-s)}\!
\int_{B_R}\!\exp\bigg (\!-p\Theta(t-s)
\frac{(|x-y|+\xi_2^+(t-s))^2}{4\kappa_0(t-s)^2(p-1)}\bigg )\mu_t(dy)-\xi_0^p\\
\ge & |({\mathcal N}(t,s)f)(x)|^pe^{-p\phi(t-s)}\!
\exp\bigg (\!-p\Theta(t-s)\frac{(|x|+R+\xi_2^+(t-s))^2}{4\kappa_0(t-s)^2(p-1)}\bigg )\mu_t(B_R)
-\xi_0^p,
\end{align*}
where $\phi=1+\xi_1^+$.
By the tightness of the family $\{\mu_t: t\in I\}$ we can fix
$R>0$ such that $\mu_t(B_R)\ge 2^{-p}$ for any $t\ge s$ and, from the previous chain of inequalities, we conclude that
\begin{align}
|({\mathcal N}(t,s)f)(x)|^p\le 2^p(\widetilde C(t-s))^p\varphi_{p\Lambda(t-s)}(x)
(\|f\|_{L^p(\Rd,\mu_s)}^p+\xi_0^p),
\label{harnack-1}
\end{align}
where
\begin{align*}
\Lambda(r)=\exp\bigg (\frac{\Theta(r)}{2\kappa_0r^2(p-1)}\bigg ),\qquad\;\,
\widetilde C(r)=\exp\bigg (\phi r+\Theta(r)\frac{(\xi_2^+r+R)^2}{2\kappa_0r^2(p-1)}\bigg ).
\end{align*}

Now, using the evolution law  and again \eqref{cavo}, we can write
\begin{align}
|({\mathcal N}(t,s)f)(x)|^p=& |({\mathcal N}(t,(t+s)/2){\mathcal N}((t+s)/2,s))(x)|^p\notag\\
\le & [(G(t,(t+s)/2)|{\mathcal N}((t+s)/2,s)f|^p)(y)+\xi_0^p]\notag\\
&\qquad\times
\exp\bigg (p\phi\frac{t-s}{2}+p\Theta\bigg (\frac{t-s}{2}\bigg )
\frac{(2|x-y|+\xi_2^+(t-s))^2}{4\kappa_0(t-s)^2(p-1)}\bigg )
\label{lara}
\end{align}
for any $y\in\Rd$. From \eqref{repres-formula} and \eqref{harnack-1} we obtain
\begin{align}
&(G(t,(t+s)/2)|{\mathcal N}((t+s)/2,s)f|^p)(y)\notag\\
\le &
2^p(\widetilde C((t-s)/2))^p(\|f\|_{L^p(\Rd,\mu_s)}^p+\xi_0^p)
(G(t,(t+s)/2)\varphi_{p\Lambda(t-s)/2})(y)\notag\\
\le &
2^p(\widetilde C((t-s)/2))^p(\|f\|_{L^p(\Rd,\mu_s)}^p+\xi_0^p)
M_{\frac{t-s}{2},p\Lambda(t-s)/2}.
\label{laretta}
\end{align}
From \eqref{lara}, \eqref{laretta}, choosing $y=x$ in the exponential term, we get
\begin{align*}
|({\mathcal N}(t,s)f)(x)|\le &
[2\widetilde C((t-s)/2)(\|f\|_{L^p(\Rd,\mu_s)}+\xi_0)
M_{\frac{t-s}{2},p\Lambda(t-s)/2}^{1/p}+\xi_0]\notag\\
&\qquad\times
\exp\bigg (\phi\frac{t-s}{2}+\Theta\bigg (\frac{t-s}{2}\bigg )
\frac{(\xi_2^+)^2}{4\kappa_0(p-1)}\bigg )
\end{align*}
and \eqref{giovanni} follows with
\begin{align*}
c_4(r)=&2\widetilde
C(r/2)M_{r/2,p\Lambda r/2}^{1/p}
\exp[(1+\xi_1^+)r/2+\Theta(r/2)(\xi_2^+)^2(4\kappa_0(p-1))^{-1}],\\
c_5(r)=&(2\widetilde
C(r/2)M_{r/2,p\Lambda r/2}^{1/p}+1)\exp[(1+\xi_1^+)r/2+\Theta(r/2)(\xi_2^+)^2(4\kappa_0(p-1))^{-1}],
\end{align*}

{\em Step 2.} We fix $t>s$, $f\in C^1_b(\Rd)$. By Theorem
\ref{exi_cb} ${\mathcal N}(t,s)f\in C^1_b(\Rd)$ and, by Step 1,
$\|{\mathcal N}((t+s)/2,s)f\|_{\infty}\le c_4((t-s)/2)\|f\|_{L^p(\Rd,\mu_s)}+c_5((t-s)/2)\xi_0$.
Hence, from \eqref{ciclismo} we get
\begin{align*}
&\sqrt{\frac{t-s}{2}}\|\nabla_x{\mathcal N}(t,s)f\|_{\infty}\\
\le&\widetilde C_{(t-s)/2}\bigg [c_4((t-s)/2)\|f\|_{L^p(\Rd,\mu_s)}+c_5((t-s)/2)\xi_0
+\frac{t-s}{2}\xi_0+\xi_0\bigg ].
\end{align*}
Taking $T=t$, estimate \eqref{mancarella} follows
with $c_6(r)=\sqrt{2}r^{-1/2}\widetilde C_{r/2}c_4(r/2)$ and $c_7(r)=\widetilde C_{r/2}[c_5(r/2)\sqrt{2/r}+\sqrt{r/2}+\sqrt{2/r}]$.
\end{proof}

\section{Stability of the null solution}

In this section we study the stability of the null solution to problem \eqref{n-sm-pb} both
in the $C_b$- and $L^p$-settings. For this reason, we assume that $\psi(\cdot,\cdot,0,0)=0$.

\begin{thm}
The following properties are satisfied.
\begin{enumerate}[\rm (i)]
\item
Let Hypotheses $\ref{base}$, $\ref{cond-iper}(i)$-$(iii)$. Further, suppose that
the constant $\omega_p=\xi_1+(\xi_2^+)^2(4\kappa_0(p-1))^{-1}$ is negative, where $\xi_1$
and $\xi_2$ are defined in Hypothesis $\ref{cond-iper}(ii)$.
 Then, for any $p\ge p_0$, there exists a positive constant $K_p$
 such that, for any $s\in I$, $f\in L^p(\Rd,\mu_s)$ and $j=0,1$,
\begin{align}
&\|D_x^j{\mathcal N}(t,s)f\|_{L^p(\Rd,\mu_t)}\leq K_p^je^{\omega_p(t-s)}\|f\|_{L^p(\Rd,\mu_s)},
\qquad\;\,t>s+j.
\label{cartellino}
\end{align}
\item
Suppose that the assumptions of Theorem $\ref{prop-3.8}$ are satisfied. Further, assume
that Hypotheses  $\ref{cond-iper}(i)$-$(iii)$ hold with $\xi_1<0$.
Then, \eqref{cartellino} holds true for any $f\in C_b(\Rd)$ with $p=+\infty$ and $\omega_p$ and $K_p$ being replaced, respectively, by $\xi_1$
and $C_1e^{-\xi_1}$.
\end{enumerate}
\end{thm}

\begin{proof}
(i) Estimate \eqref{cartellino} can be obtained
arguing as in the proof of Theorem \ref{vacanza}, where now $p(t)=p$ for any $t\ge s$.
As far as the gradient of ${\mathcal N}(t,s)f$ is concerned, we fix $t>s+1$ and observe that
${\mathcal N}(t,s)f={\mathcal N}(t,t-1){\mathcal N}(t-1,s)f$.
Hence, from \eqref{ciclismo} we obtain
\begin{eqnarray*}
\|\nabla_x{\mathcal N}(t,s)f\|_{L^p(\Rd,\mu_t)}\le C_1\|{\mathcal N}(t-1,s)f\|_{L^p(\Rd,\mu_{t-1})}
\le K_pe^{\omega_p(t-s)}\|f\|_{L^p(\Rd,\mu_s)},
\end{eqnarray*}
where $K_p=C_1e^{-\omega_p}$.

(ii) The assertion follows easily letting $p$ tend to $+\infty$ in \eqref{cartellino}.
\end{proof}

\appendix

\section{Technical results}
\begin{prop}\label{smoth_v}
Let Hypotheses $\ref{base}$ hold and let $g\in C((a,b]\times \Rd)$ satisfy
$[g]_{\gamma,\infty}:=\sup_{r\in (a,b)}(r-a)^{\gamma}\|g(r,\cdot)\|_\infty<+\infty$ for some $\gamma\in [0,1)$ and some $I\ni a<b$. Then, the function $z:[a,b]\times\Rd\to\R$, defined by
\begin{equation*}
z(t,x):=\int_a^t (G(t,r)g(r, \cdot))(x)dr, \qquad t \in [a,b], \, x \in \Rd,
\label{x-men}
\end{equation*}
belongs to $C_b([a,b]\times \Rd)\cap C^{0,1+\theta}((a,b]\times \Rd)$ for any $\theta\in (0,1)$,
\begin{equation}
\|z\|_{\infty}\le \frac{(b-a)^{1-\gamma}}{1-\gamma}[g]_{\gamma,\infty},\qquad\;\,\|\nabla_xz(t,\cdot)\|_{\infty}\le c_{\gamma,a,b}(t-a)^{\frac{1}{2}-\gamma}[g]_{\gamma,\infty}
\label{leopolda}
\end{equation}
and
\begin{equation}
\|\nabla_x z(t,\cdot)\|_{C^{\theta}(B_R)}\le C_R[g]_{\gamma,\infty}(t-a)^{\frac{1-2\gamma-\theta}{2}},
\label{leopolda-1}
\end{equation}
for any $t\in (a,b]$, $R>0$ and some positive constants $c_{\gamma,a,b}$ and $C_R$.
In particular, if $\gamma\le 1/2$, then $\nabla_xz$ is bounded in $(a,b]\times\Rd$.

Finally, if $[g]_{\gamma,\theta,R}:=\sup_{t\in (a,b]}(t-a)^{\gamma}\|g(t,\cdot)\|_{C^{\theta}_b(B_R)}<+\infty$, for some $\theta\in (0,1)$ and any $R>0$, then
$z\in C^{0,2+\theta}_{\rm loc}((a,b]\times\Rd)\cap C^{1,2}((a,b]\times\Rd)$. Moreover,
\begin{align}
\|z(t,\cdot)\|_{C_b^{2}(B_R)}\le c(t-a)^{\frac{\theta}{2}-\gamma}[g]_{\gamma,\theta,R+1},\qquad\;\,t\in (a,b].
\label{rai-sport}
\end{align}
and, if $\theta>\alpha$, then
\begin{align}
\|z(t,\cdot)\|_{C_b^{2+\rho}(B_R)}\le c(t-a)^{\frac{\theta-\rho}{2}-\gamma}[g]_{\gamma,\theta,R+1},\qquad\;\,t\in (a,b],
\label{rai-sport-1}
\end{align}
where $\rho=\alpha$ if $\theta>\alpha$, whereas $\rho$ can be arbitrarily fixed in $(0,\theta)$ otherwise
\end{prop}

\begin{proof}
Throughout the proof, we will make use of \cite[Proposition 2.7]{AL}, where it has been shown that, for any $I\ni a<b$,
$R>0$, $\eta\in (0,1]$ and $\beta\in [\eta,2+\alpha]$ there exists positive constants $C_{\beta}=C_{\beta}(a,b,R)$ and $C_{\eta,\beta}=C_{\eta,\beta}(a,b,R)$ such that for any $f\in C_b(\Rd)\cap C^{\eta}_{\rm loc}(\Rd)$
\begin{equation}
\label{28AL}
\|G(t,s)f\|_{C^{\beta}(\overline{B}_R)} \leq \left\{
\begin{array}{ll}
C_{\beta}(t-s)^{-\frac{\beta}{2}}\|f\|_{\infty},\\[2mm]
C_{\eta,\beta}(t-s)^{-\frac{\beta-\eta}{2}}\|f\|_{C^{\eta}(\overline{B}_{R+1})},
\end{array}
\right.
\qquad\;\, a\leq s<t\leq b.
\end{equation}

To begin with, we observe that, for any $t\in (a,b]$ and $x\in\Rd$, the function $r\mapsto (G(t,r)g(r,\cdot))(x)$ is measurable in $(a,t]$.
If $g$ is bounded and uniformly continuous in $\R^{d+1}$ this is clear. Indeed, as it has been recalled in Section \ref{sect-2}, the function $(t,s,x)\mapsto (G(t,s)f)(x)$ is continuous in $\{(t,s,x)\in I\times I\times\Rd: t\ge s\}$
for any $f\in C_b(\Rd)$. Hence, taking \eqref{norm_g} into account and adding and subtracting $(G(t,r)g(r_0,\cdot))(x)$,
we can estimate
\begin{align*}
&|(G(t,r)g(r,\cdot))(x)-(G(t,r_0)g(r_0,\cdot))(x_0)|\\
\le & \|g(r,\cdot)-g(r_0,\cdot)\|_{\infty}+|(G(t,r)g(r_0,\cdot))(x)-(G(t,r_0)g(r_0,\cdot))(x_0)|
\end{align*}
for any $(r,x), (r_0,x_0)\in [a,t]\times\Rd$, and the last side of the previous chain of inequalities vanishes as $(r,x)$ tends to $(r_0,x_0)$.

If the function $g$ is as in the statement of the proposition, we can approximate it by a sequence $(g_n)$ of bounded and uniformly continuous functions in $\R^{d+1}$ which converge to $g$ pointwise in $(a,b)\times\Rd$
and satisfy $\|g_n(r,\cdot)\|\le \|g(r,\cdot)\|_{\infty}$ for any $r\in (a,b)$.\footnote{This can be done, for instance, setting $g_n(t,x)=\vartheta_n(t)(\overline g(t,\cdot)\star \rho_n)(x)$ for any $(t,x)\in\R^{d+1}$ and $n\in\N$, where
$\overline g:(a,+\infty)\times\Rd\to\R$ equals $g$ in $(a,b)\times\Rd$ and $\overline g(t,\cdot)=g(b,\cdot)$ for any $t>b$,
$(\vartheta_n)\subset C^{\infty}(\R)$ is a sequence of smooth functions such that $\one_{[a+2/n,+\infty)}\le\vartheta_n\le \one_{[a+1/n,+\infty)}$ for any $n\in\N$
and ``$\star$'' denotes convolution with respect to the spatial variables.}
Since the sequence $(g_n)$ is bounded and pointwise converges to $g$ in $(a,t]\times\Rd$, by \cite[Proposition 3.1(i)]{KunLorLun09Non}
$(G(t,\cdot)g_n(r,\cdot))(x)$ converges to $(G(t,\cdot)g(r,\cdot))(x)$ as $n\to +\infty$ pointwise in $(a,t]$.
Hence, the function $r\mapsto (G(t,r)g(r,\cdot))(x)$ is measurable in $(a,t]$.

Using again \eqref{norm_g} we obtain
$\|G(t,r)g(r,\cdot)\|_{\infty}\le \|g(r,\cdot)\|_{\infty}\le (r-a)^{-\gamma}[g]_{\gamma,\infty}$ for any $r\in (a,t]$. It thus follows that $z$ is bounded and
the first estimate in \eqref{leopolda} follows.

Proving that $z$ is continuous in $[a,b]\times\Rd$ is an easy task, based on estimate \eqref{norm_g} and the dominated convergence theorem. Hence, the details are omitted.

Fix $\theta\in (0,1)$. The first estimate in \eqref{28AL} with $\beta=1+\theta$ and the assumptions on $g$ allow to differentiate
$z$ with respect to $x_j$ ($j=1,\ldots,d$), under the integral sign, and obtain that $D_j z(t,\cdot)$ is locally $\theta$-H\"older continuous in $\Rd$, uniformly with respect to $t \in (a,b)$, and
\begin{equation}
\|D_jz(t,\cdot)\|_{C^{\theta}(B_R)}\le C_R[g]_{\gamma,\infty}(t-a)^{\frac{1-2\gamma-\theta}{2}},\qquad\;\,t\in (a,b].
\label{sgura}
\end{equation}
To conclude that $D_j z$ is continuous in $(a,b]\times \Rd$, it suffices to prove that, for any $x\in\Rd$, the function $D_jz(\cdot,x)$ is continuous in $(a,b]$.
For this purpose, we apply an interpolation argument. We fix $R>0$ such that $x\in \overline{B}_R$. Applying the well-known interpolation estimate
$\|f\|_{C^1(\overline{B}_R)}\le K\|f\|_{C(\overline{B}_R)}^{\theta/(1+\theta)}\|f\|_{C^{1+\theta}(B_R)}^{1/(1+\theta)}$
with $f=z(t,\cdot)-z(t_0,\cdot)$ and $t,t_0\in (a,b]$, from the continuity of $z$ in $[a,b]\times\Rd$ and the local boundedness in $(a,b]$ of the function $t\mapsto\|f(t,\cdot)\|_{C^{1+\theta}(B_R)}$,
we conclude that the function $D_jz(\cdot,x)$ is continuous in $(a,b]$. Hence, $z\in C^{0,1+\theta}_{\rm loc}((a,b]\times\Rd)$. Estimate \eqref{leopolda-1} follows from \eqref{sgura}.
Further, estimate \eqref{point_p_D} and the assumption on $g$ imply that
\begin{align*}
|D_jz(t,x)|\le C_0[g]_{\gamma,\infty}\int_a^t(r-a)^{-\gamma}(1+(t-r)^{-1/2})dr=C'_{\gamma,a,b}(t-a)^{\frac{1}{2}-\gamma}[g]_{\gamma,\infty}
\end{align*}
for any $(t,x)\in (a,b]\times\Rd$, whence the second estimate in \eqref{leopolda} follows at once.

Let us now assume that $\sup_{t\in (a,b)}(t-a)^{\gamma}\|g(t,\cdot)\|_{C^{\theta}_b(B_R)}<+\infty$ for any $R>0$. Arguing as above and
taking the second estimate in \eqref{28AL} with $\beta=2$ (resp. $\beta=2+\alpha$) into account,  we can show that $z(t,\cdot)\in C^2_{\rm loc}(\Rd)$ (resp.
$z(t,\cdot)\in C^{2+\alpha}_{\rm loc}(\Rd)$) for any $t \in (a,b]$ and
\eqref{rai-sport} (resp. \eqref{rai-sport-1}) holds true. Applying the interpolation  inequality
$\| \varphi\|_{C^2(\overline{B}_R)} \leq C \|\varphi\|_{\infty}^{\theta/(2+\theta)} \| \varphi\|_{C^{2+\theta} (\overline{B}_R)}^{2/(2+\theta)}$
with $\varphi=z(t,\cdot)-z(t_0,\cdot)$ we deduce that the second-order spatial derivatives of $z$ are continuous in $(a,b]\times B_R$ and, hence, in $(a,b]\times\Rd$ due to the arbitrariness of $R>0$.

Finally, to prove the differentiability of $z$, we introduce the sequence $(z_n)$, where
\begin{eqnarray*}
z_n(t,x)=\int_a^{t-\frac{1}{n}}(G(t,r)g(r, \cdot))(x)dr, \qquad t \in [a+1/n,b],\;\, x \in \Rd,\;\,n\in\N.
\end{eqnarray*}
As it is immediately seen $z_n$ converges to $z$, locally uniformly in $(a,b]\times\Rd$ and each function $z_n$ is differentiable in $[a+1/n,b]\times\Rd$ with respect to $t$ and
\begin{eqnarray*}
D_tz_n(t,x)=\int_a^{t-\frac{1}{n}}(\A(t)G(t,r)g(r, \cdot))(x)dr+(G(t,t-1/n)g(t-1/n,\cdot))(x)
\end{eqnarray*}
for such values of $(t,x)$. Since $\|\A(t)G(t,r)g(r, \cdot)\|_{C_b(B_R)}\le C_R[g]_{\gamma,\infty}(t-r)^{\theta/2-\gamma}(r-a)^{-\gamma}$ for any  $r\in (a,t)$, and $g(t-1/n,\cdot)$ converges to
$g(t,\cdot)$ locally uniformly in $\Rd$, by \cite[Proposition 3.6]{KunLorLun09Non} and the dominated convergence theorem, we conclude that
$D_tz_n$ converges locally uniformly in $(a,b]\times\Rd$ to $\A z+g$. Thus, we conclude that $z$ is continuously differentiable
in $(a,b]\times\Rd$ and, therein, $D_tz=\A z+g$.
\end{proof}

\begin{lemm}
\label{lemm-brigida}
Let $J$ be an interval and let $g\in C(J\times\Rd)$ be such that $g(t,\cdot)$ is bounded in $\Rd$ for any $t\in J$.
Then, the function $t\mapsto \|g(t,\cdot)\|_{\infty}$ is measurable in $J$.
\end{lemm}

\begin{proof}
To begin with, we observe that for any $n\in\N$ the function $t\mapsto \|g(t,\cdot)\|_{C(\overline B_n)}$ is continuous in $J$.
This is a straightforward consequence of the uniform continuity of $g$ in $J_0\times B_n$ for any bounded interval $J_0$
compactly embedded into $J$.
To complete the proof, it suffices to show that $z_n(t):=\|g(t,\cdot)\|_{C(\overline B_n)}$ converges to $\|g(t,\cdot)\|_{\infty}$
for any $t\in J$. Clearly, for any fixed $t\in J$, the sequence $(z_n(t))$ is increasing and is bounded from
above by $\|g(t,\cdot)\|_{\infty}$.
To prove that, $(z_n(t))$ converges to $\|g(t,\cdot)\|_{\infty}$,
we fix a sequence $(x_n)\subset\Rd$ such that
$|g(t,x_n)|$ tends to $\|g(t,\cdot)\|_{\infty}$ as $n\to +\infty$.
For any $n\in\N$, let $k_n\in\N$ be such that $x_n\in B_{k_n}$. Without loss of generality, we can assume that the
sequence $(k_n)$ is increasing. Then,
$z_{k_n}(t)=\|g(t,\cdot)\|_{C(\overline B_{k_n})}\ge |g(t,x_n)|$ for any $n\in\N$.
Hence, the sequence $(z_{k_n}(t))$ converges to $\|g(t,\cdot)\|_{\infty}$ and this is enough to conclude that
the whole sequence $(z_n(t))$ converges to $\|g(t,\cdot)\|_{\infty}$ as $n\to +\infty$.
\end{proof}

Finally, we prove some interior $L^p$-estimates.

\begin{prop}
\label{prop-A3}
Let $\Omega\subset\Rd$ be a bounded open set and let
$u\in C^{1,2}((s,T)\times\Omega)$ solve the equation $D_tu=\A u$ in $(s,T)\times \Omega$.
Then, for any $x_0\in\Omega$ and $R_1>0$, such that $B_{R_1}(x_0)\Subset \Omega$, there
exists a positive constant $c=c(R_1,x_0,s,T)$ such that
\begin{align*}
(t-s)\|u(t,\cdot)\|_{W^{2,p}(B_{R_1}(x_0))}+\sqrt{t-s}\|u(t,\cdot)\|_{W^{1,p}(B_{R_1}(x_0))}
\leq c\sup_{r\in(s,T)}\|u(r,\cdot)\|_{L^p(\Omega)}.
\end{align*}
\end{prop}

\begin{proof}
Throughout the proof, we denote by $c$ a positive constant, independent of $n$ and $u$, which may vary from line to line.

Let us fix $0<R_1<R_2$, such that $\overline{B_{R_2}(x_0)}\subset\Omega$, and
a sequence of cut-off functions $(\vartheta_n)\subset C_c^\infty(\Omega)$
such that $\one_{B_{r_n}(x_0)}\leq\vartheta_n\leq \one_{B_{r_{n+1}}(x_0)}$ and
$\|\vartheta_n\|_{C_b^k(\Omega)}\leq 2^{kn}c$, for any $n\in\N\cup\{0\}$ and $k=0,1,2,3$, where
$r_n:=2R_1-R_2+(2-2^{-n})(R_2-R_1)$. Since the function $u_n:=\vartheta_nu$ solves the equation $D_tu_n=\A u_n+g_n$ in $(s,T)\times B_{r_{n+1}}(x_0)$, where
$g_n=-u\A\vartheta_n-\langle Q\nabla_xu,\nabla\vartheta_n\rangle$, we can write
\begin{align}
u_n(t,x)=(G_{n+1}^{{\mathcal D}}(t,s)\vartheta_nu(s,\cdot))(x)+\int_s^t(G_{n+1}^{{\mathcal D}}(t,\sigma)g_n(\sigma,\cdot))(x)d\sigma,
\label{acqua-5}
\end{align}
where $G_{n+1}^{{\mathcal D}}(t,s)$ is the evolution operator associated to the realization of the operator $\A$ in
$L^p(B_{r_{n+1}}(x_0))$ with homogeneous Dirichlet boundary conditions. It is well known that
$\|G^{{\mathcal D}}(t,r)\psi\|_{W^{2,p}(B_{r_{n+1}}(x_0))}\leq c(t-r)^{-1+\frac{\alpha}{2}}\|\psi\|_{W^{\alpha,p}(B_{r_{n+1}}(x_0))}$
for any $\alpha\in (0,1)$, $\psi\in W^{\alpha,p}(B_{r_{n+1}}(x_0))$ and $s\leq r<t\leq T$. Since $g_n(\sigma,\cdot)\in W^{\alpha,p}(B_{r_{n+1}}(x_0))$ for any $\sigma\in(s,t)$, from \eqref{acqua-5} we obtain
\begin{align*}
(t-s)\|u(t,\cdot)\|_{W^{2,p}(B_{r_n}(x_0))}\le &c\|u(s,\cdot)\|_{L^p(B_{r_{n+1}}(x_0))}\\
&+c\int_s^t (t-\sigma)^{-1+\frac{\alpha}{2}}\|g_n(\sigma,\cdot)\|_{W^{\alpha,p}(B_{r_{n+1}}(x_0))}d\sigma.
\end{align*}
Now, for any $n\in\N$ we set $\zeta_n:=\sup_{t\in(s,T)}(t-s)\|u(t,\cdot)\|_{W^{2,p}(B_{r_n}(x_0))}$ and estimate the function under the
integral sign. At first, we note that
\begin{align*}
\|g_n(\sigma,\cdot)\|_{W^{\alpha,p}(B_{r_{n+1}}(x_0))}
\leq c\|\vartheta_n\|_{C_b^{2+\alpha}(B_{r_{n+1}}(x_0))}\|u(\sigma,\cdot)\|_{W^{1+\alpha,p}(B_{r_{n+1}}(x_0))}.
\end{align*}
By interpolation and using Young's inequalities we obtain, for any $\sigma\in (s,t)$,
\begin{align*}
\|u(\sigma,\cdot)\|_{W^{1,p}(B_{r_{n+1}}(x_0))}
\leq & c(\sigma-s)^{-\frac{1}{2}}\|u(\sigma,\cdot)\|^{\frac{1}{2}}_{L^p(B_{r_{n+1}}(x_0))}\sqrt{\zeta_{n+1}} \\
\leq & (\sigma-s)^{-\frac{1}{2}}\left(c\varepsilon^{-1}\|u(\sigma,\cdot)\|_{L^p(\Omega)}+\varepsilon\zeta_{n+1}\right),
\end{align*}
and
\begin{align*}
\|\nabla_xu(\sigma,\cdot)\|_{W^{\alpha,p}(B_{r_{n+1}}(x_0))}
\leq & (\sigma-s)^{-\frac{1+\alpha}{2}}\!\Big (c\varepsilon^{-\frac{1+\alpha}{1-\alpha}}\|u(\sigma,\cdot)\|_{L^p(\Omega)}
+\varepsilon\zeta_{n+1}\Big ).
\end{align*}
Collecting the above estimates together we get
\begin{eqnarray*}
\zeta_n\leq 8^nc\varepsilon\zeta_{n+1}+c\sup_{r\in(s,T)}\|u(r,\cdot)\|_{L^p(\Omega)}(1+8^n\varepsilon^{-(1+\alpha)/(1-\alpha)}).
\end{eqnarray*}
Now we fix $0<\eta<64^{-1/(1+\alpha)}$ and $\varepsilon=8^{-n}c^{-1}\eta$. Multiplying both
the sides of the previous inequality by $\eta^n$ and summing up from $0$ to $N$ yields
\begin{align}
\zeta_0-\eta^{N+1}\zeta_{N+1}
\leq c\sup_{r\in(s,T)}\|u(r,\cdot)\|_{L^p(\Omega)}.
\label{reti}
\end{align}
Since
$\{\zeta_n\}_{n\in\N}$ is bounded, taking the limit as $N\rightarrow+\infty$ in the
left-hand side of \eqref{reti} we conclude that
$(t-s)\|u(t,\cdot)\|_{W^{2,p}(B_{R_1}(x_0))}\leq c\sup_{r\in(s,T)}\|u(r,\cdot)\|_{L^p(\Omega)}$
for any $t\in(s,T)$. An interpolation argument gives
$\|u(t,\cdot)\|_{W^{1,p}(B_{R_1}(x_0))}\leq c(t-s)^{-1/2}\sup_{r\in(s,T)}\|u(r,\cdot)\|_{L^p(\Omega)}$ for any
$t\in(s,T)$, and this completes the proof.
\end{proof}


\begin{thebibliography}{99}

\bibitem{AngLor10Com}
L. Angiuli, L. Lorenzi,
\newblock{\em Compactness and invariance properties of evolution operators
associated to Kolmogorov operators with unbounded coefficients},
\newblock{ J. Math. Anal. Appl.} {\bf 379} (2011), 125--149.

\bibitem{AngLorOnI}
L. Angiuli, L. Lorenzi,
\newblock{\em On improvement of summability properties in nonautonomous Kolmogorov equations},
\newblock{Comm. Pure Appl. Anal.} {\bf 13} (2014), 1237--1265.

\bibitem{AngLorLun}
L. Angiuli, L. Lorenzi, A. Lunardi,
\newblock{\em Hypercontractivity and asymptotic behaviour in nonautonomous Kolmogorov equations},
\newblock{Commun. Partial. differential equations} {\bf 28} (2013), 2049--2080.

\bibitem{AL}
L. Angiuli, A. Lunardi,
\newblock{\em Semilinear nonautonomous parabolic equations with unbounded coefficients in the linear part}
\newblock{Nonlinear Analysis}, {\bf 125}, (2015), 468--497.

\bibitem{DaPRoc}
G. Da Prato, M. R{\"o}ckner,
\newblock{\em A note on evolution systems of measures for time-dependent stochastic differential equations},
\newblock{in Seminar on Stochastic Analysis, Random {F}ields and Applications V,}
\newblock{Progr. Probab.}, \textbf{59}, Birkh\"auser, Basel, 2008, 115-122.


\bibitem{Gronwall}
T.H. Gronwall,
\newblock{\em Note on the derivation with respect to a parameter of the solutions of a system of differential equations}\,,
\newblock{Ann. Math. \textbf{20} (1919)}, 292-296.

\bibitem{henry}
D. Henry,
\newblock{Geometric theory of semilinear parabolic equations},
\newblock{Lecture Notes in Mathematics}, \textbf{840},
\newblock{Springer-Verlag, Berlin-New York}, 1981.

\bibitem{KunLorLun09Non}
M. Kunze, L. Lorenzi, A. Lunardi,
\newblock {\em Nonautonomous Kolmogorov parabolic equations
with unbounded coefficients},
\newblock {Trans. Amer. Math. Soc.} {\bf{362}} (2010), 169-198.


\bibitem{LorLibro-2}
L. Lorenzi,
\newblock
Analytical methods for Kolmogorov equations, Second edition,
\newblock{CRC Press, Taylor \& Francis Group}, 2016 (to appear).

\bibitem{LorLunSch16Str}
L. Lorenzi, A. Lunardi, R. Schnaubelt,
\newblock {\em Strong convergence of solutions to nonautonomous Kolmogorov equations},
\newblock
{Proc. Amer. Math. Soc.} {\bf 144} (2016), 3903-3917.

\bibitem{LorLunZam10Asy}
L. Lorenzi, A. Lunardi, A. Zamboni,
\newblock
{\em Asymptotic behavior in time periodic parabolic problems with unbounded coefficients},
\newblock
{J. Differential Equations} \textbf{249} (2010), 3377-3418.


\bibitem{LorZam}
L. Lorenzi, A. Zamboni,
\newblock {\em Cores for parabolic operators with unbounded coefficients},
\newblock
{J. Differential Equations} \textbf{246} (2009), 2724-2761.

\end{thebibliography}
\end{document}